\documentclass[a4paper,12pt]{amsart}
\usepackage{amsmath}
\usepackage{amssymb}
\usepackage{amsthm}
\usepackage{amsfonts}
\usepackage[francais, english]{babel}

\DeclareMathSymbol{\subsetneqq}{\mathbin}{AMSb}{36}

\newcommand{\R}{\mathbb{R}}

\newcommand{\N}{\mathbb{N}}

\newcommand{\C}{\mathbb{C}}

\newcommand{\beq}{\begin{eqnarray}}
\newcommand{\eeq}{\end{eqnarray}}
\newcommand{\bq}{\begin{equation}}
\newcommand{\eq}{\end{equation}}
\newcommand{\beqn}{\begin{eqnarray*}}
\newcommand{\eeqn}{\end{eqnarray*}}
\newcommand{\bex}{\begin{exo}}
\newcommand{\eex}{\end{exo}}
\newcommand{\ben}{\begin{enumerate}}
\newcommand{\een}{\end{enumerate}}

\newtheorem{th1}{{\bf Theorem}}[section]
\newtheorem{thm}[th1]{{\bf Theorem}}
\newtheorem{lem}[th1]{{\bf Lemma}}
\newtheorem{prop}[th1]{{\bf Proposition}}

\newtheorem{rem}[th1]{\bf Remark}

\newtheorem{defi}[th1]{\bf Definition}

\author[T. Saanouni]{T. Saanouni}
\address{University of Tunis El Manar, Faculty of Science of Tunis, LR03ES04 partial differential Equations and applications, 2092 Tunis, Tunisia.}
\email{\sl Tarek.saanouni@ipeiem.rnu.tn}
\thanks{{\sf T. Saanouni} is grateful to the Laboratory of
PDE and Applications at the Faculty of Sciences of Tunis.}


\title[Global well-posedness and instability...]{Global well-posedness and instability of an inhomogeneous nonlinear Schr\"odinger equation with harmonic potential}

\date{\today}
\begin{document}
\begin{abstract}
This paper is concerned with the Cauchy problem for an inhomogeneous nonlinear Schr\"odinger equation with exponential growth nonlinearity and harmonic potential in two space dimensions. We prove global well-posedness, existence of the associated ground state and instability of the standing wave.
\end{abstract}


\maketitle
\tableofcontents
\vspace{ 1\baselineskip}
\renewcommand{\theequation}{\thesection.\arabic{equation}}
\section{Introduction}
Consider the initial value problem for a nonlinear Schr\"odinger equation
\begin{equation}\label{eq1}
\left\{
\begin{matrix}
i\dot u+\Delta u-|x|^2 u+ \epsilon |x|^\mu g(u)=0,\\
u_{|t=0}= u_0,
\end{matrix}
\right.
\end{equation}
where $\mu>0,\epsilon\in\{-1,1\}$ and $u$ is a complex-valued function of the variable $(t,x)\in \mathbb{R}\times {\mathbb{R}}^2$. The nonlinearity takes the Hamiltonian form $g(u):=uG'(|u|^2)$ for some regular positive real function $G$.\\
A solution $u$ of $\eqref{eq1}$, formally satisfies the conservation of mass and energy
\begin{gather}
M(t)= M(u(t)):=\|u(t)\|_{L^2}^2=M(0),\nonumber\\
E(t)= E(u(t)):=\|\nabla u(t)\|_{L^2}^2+\|xu\|_{L^2}^2-\epsilon \int_{\R^2} |x|^\mu G(|u(t)|^2)\, dx=E(0),\nonumber\\
\frac{1}8(\|xu(t)\|^2)''=\|\nabla u\|_{L^2}^2-\|xu\|_{L^2}^2-\epsilon\int_{\R^2}|x|^\mu\Big(\bar ug(u)-(1+\frac\mu2)G(|u|^2)\Big)dx.\label{vrl}
\end{gather}
The last equality is called Virial identity \cite{Cas}.
If $\epsilon=-1$, the energy is always positive and we say that \eqref{eq1} is defocusing. Else, \eqref{eq1} is said to be focusing. \\

Equation \eqref{eq1} models the Bose-Einstein condensates with the attractive inter-particle interactions under a magnetic trap \cite{bsh,fs,lp,tw,z}. The isotropic harmonic potential $|x|^2$ describes a magnetic field whose role is to confine the movement of particles \cite{bsh,fs,tw}.
In the monomial homogeneous case $\mu=0$ and $g(u)=u|u|^{p-1}$, for $1<p<\frac{n+2}{n-2}$ if $n\geq 3$ and $1<p<\infty$ if $n\in\{1,2\}$, local well-posedness in the conformal space was established \cite{o,Cas}.
By \cite{rm}, when $p<1+\frac4n$ or $p\geq 1+\frac4n$ and $\epsilon=-1$ the solution exists globally. For $p = 1 +\frac4n$, there exists a sharp condition \cite{zh}
of the global existence. When $p >1+\frac4n$, the solution blows up in a finite time for a class of sufficiently large data and globally exists for a class of sufficiently small data \cite{rm1,rm2,tw}.\\

 In two space dimensions, the semilinear monomial Schr\"odinger problem (\eqref{eq1} with $\mu=0$ and $g(u)=u|u|^{p-1}$) is energy subcritical for all $p>1$ \cite{G.V,Cas.F}. So it's natural to consider problems with exponential nonlinearities, which have several applications, as for example the self trapped beams in plasma \cite{Lam}. Moreover, the two dimensional case is interesting because of its relation to the critical Moser-Trudinger inequalities \cite{Ad,Ru}.\\

 The two dimensional semilinear Schr\"odinger problem with exponential growth nonlinearities was studied, for small Cauchy data in \cite{Na}, global well-posedness and scattering was proved. Later on global well-posedness and scattering in the defocusing sign was obtained for some critical case \cite{Col.I,I}, which is related to the data size. The author \cite{T2} obtained a decay result in the critical case. Recently, the author \cite{T} proved global well-posedness and scattering, without any condition on the data, of some defocusing semilinear Schr\"odinger equation with exponential nonlinearity (similar results was proved for corresponding wave equation \cite{OMTS,OMTS1}).\\
The focusing case is related to so called ground state \cite{sl,ww}. Indeed, when the data energy is less than the ground state one, the solution either blows-up in finite time or exists for any time \cite{km}. In this case the stability of standing waves is a natural question \cite{T3}. \\

 In the non homogeneous case, existence and nonexistence of blow-up solutions have been studied in \cite{mrl}, where the inhomogeneity takes the form $K(x)u|u|^{p-1}$ and $K$ is bounded. Moreover, the instability of standing waves was proved \cite{fw,ww} under some conditions on $K$. \\
Existence and nonexistence of solutions to the nonlinear Schr\"odinger problem \eqref{eq1} with $g(u)=u|u|^{p-1}$ was treated \cite{jc} in the radial case. Moreover, instability of the standing waves hold. \\
 
Our aim in this paper, is to extend in two space dimensions, results about global well-posedness, blow-up in finite time and stability of the standing waves, which hold \cite{jc} for any polynomial power, to an exponential growth nonlinearity.\\

The plan of the note is as follows. The main results and some technical tools needed in the sequel are listed in the next section. The third and fourth sections are devoted to prove well-posedness of \eqref{eq1}. The goal of the fifth section is to study the stationary problem associated to \eqref{eq1}. In the sixth section we prove either global well-posedness or blow-up in finite time of solution to \eqref{eq1} for data with energy less than the ground state one. In the last section we prove an instability result about standing waves. \\\\
In this paper, we are interested in the two space dimensions case, so, here and hereafter, we denote $\int\,.\,dx:=\int_{\R^2}\,.\,dx$. For $p\geq1$, $L^p:=L^p(\R^2)$ is the Lebesgue space endowed with the norm $\|\,.\,\|_{p}:=\|\,.\,\|_{L^p}$ and $\|\,.\,\|:=\|\,.\,\|_2$. $H^1$ is the usual Sobolev space endowed with the norm $\|\,.\,\|_{H^1}^2:=\|\,.\,\|^2+\|\,\nabla.\,\|^2$ and $H^1_{rd}$ denotes the set of radial functions in $H^1$. The conformal space is $\Sigma:=\{u\in H^1_{rd}$ s. t $\int|x|^2|u(x)|^2\,dx<\infty\}.$\\
For $T>0$ and $X$ an abstract space, we denote $C_T(X):=C([0,T],X)$ the space of continuous functions with variable in $[0,T]$ and values in $X$. We mention that $C$ is an absolute positive constant which may vary from line to line. If $A$ and $B$ are nonnegative real numbers, $A\lesssim B$ means that $A\leq CB$. Finally, we define the operator $(Df)(x):=xf'(x)$.
 \section{Background material}
In this section we give the main results and some technical tools needed in the sequel. Here and hereafter, we assume that near zero, $g\simeq r^q$ for some $q:=q_g>2+2\mu$. Let give some conditions on the nonlinearity, which will be useful along this paper. 
\begin{enumerate}
\item Ground state condition
\begin{equation}\label{f}
\mu>2,\quad q_g>3+2\mu,\quad (D-1)G>0\quad\mbox{and}\quad(D-1)^2G>0\quad\mbox{on}\quad\R_+^*.
\end{equation}
\item
Strong ground state condition
\begin{equation}\label{4}
\mu>2,\quad q_g>3+2\mu\quad\mbox{and}\quad\left\{
\begin{matrix}
\exists\varepsilon_g>0\quad\mbox{s.t}\quad (D-2-\frac\mu2-\varepsilon_g)G>0\quad\mbox{on}\quad\R_+^*,\\
(D-2-\frac\mu2)(D-1-\frac\mu2)G>0\quad\mbox{on}\quad\R_+^*.
\end{matrix}
\right.
\end{equation}
\item Subcritical case
\begin{equation}\label{sc}
\forall\alpha>0,\quad |G'''(r)|=o(e^{\alpha r})\quad\mbox{as}\quad r\rightarrow\infty.
\end{equation}
\item Critical case
\begin{equation}\label{cc}
\exists\alpha_g>0\;\mbox{ s.t}\;|G'''(r)|=O(e^{\alpha_g r})\quad\mbox{as}\quad r\rightarrow\infty.
\end{equation}
\end{enumerate}
We will say that the nonlinearity or the problem \eqref{eq1} is subcritical (respectively critical) if $G$ satisfies \eqref{sc} (respectively \eqref{cc}).
\begin{rem}\mbox{}\\
\begin{enumerate}
\item
Previous assumptions arise quite naturally, when we study the two dimensional Schr\"odinger problem \cite{Col.I,T,T3}.
\item
We should assume \eqref{sc} or \eqref{cc} in order to prove well-posedness of \eqref{eq1} and we will use $[$\eqref{f} or \eqref{4}$]$ with $[$\eqref{sc} or \eqref{cc}$]$ in order to obtain existence of a ground state for the stationary problem associated to \eqref{eq1}.
\item
We give explicit examples.
\begin{enumerate}
\item
Subcritical case:\quad $G(r):= {\rm e}^{(1+r)^{\frac{1}{2}}}-\frac e2r-e.$
\item Critical case:\quad $G(r)=e^r-1-r-\frac12r^2.$
\end{enumerate}
\end{enumerate}
\end{rem}
\begin{proof}
\begin{enumerate}
\item
For $t:=\sqrt{r+1}$, we have $G(r)=e^t-\frac e2t^2-\frac{e}2$. Thus, $DG(r)=\frac{t^2-1}{2t}(e^t-et)$. Compute, for $\varepsilon>0$,
\begin{gather*}
\phi(t):=2(D-1-\varepsilon)G(r)=e^t(t-\frac1t-2-2\varepsilon)+e(\varepsilon t^2+2+\varepsilon),\\
\phi'(t)=e^t(t-\frac1t+\frac1{t^2}-1-2\varepsilon)+2e\varepsilon t,\\
\phi''(t)=e^t(t-\frac1t+\frac2{t^2}-\frac2{t^3}-2\varepsilon)+2e\varepsilon \geq0.
\end{gather*}
Since $\phi(1)=\phi'(1)=0$, we have $\phi\geq 0$. Moreover,
\begin{gather*}
D(D-1)G(r)=\frac14e^t(t-\frac1t)(t-1-\frac1t+\frac1{t^2}),\\
(D-1)^2G(r)=\frac14[e^t(t^2-3t+2+\frac4t+\frac1{t^2}-\frac1{t^3})-4e],\\
[(D-1)^2-\varepsilon]G(r)=\frac14[e^t(t^2-3t+2-4\varepsilon+\frac4t+\frac1{t^2}-\frac1{t^3})+2e\varepsilon t^2+2\varepsilon e-4e]:=\frac14\psi(t),\\
\psi'(t)=e^t(t^2-t-1-4\varepsilon+\frac4t-\frac3{t^2}-\frac3{t^3}+\frac3{t^4})+4e\varepsilon t,\\
\psi''(t)=e^t(t^2+t-4\varepsilon+\frac4t-\frac7{t^2}+\frac3{t^3}+\frac{12}{t^4}-\frac{12}{t^5})+4e\varepsilon \geq 0.
\end{gather*}
Sine $\psi(1)=\psi'(1)=0$, we have $\psi\geq0.$
\item
Take $\varepsilon\in(0,2)$ and $G(x):=e^x-1-x-\frac{x^2}2$. Then $DG(x)=x(e^x-1-x)$,
\begin{gather*}
(D-1-\varepsilon)G(x)=(x-1-\varepsilon)e^x+(\varepsilon-1)\frac{x^2}2+\varepsilon x+1+\varepsilon:=\phi(x),\\
\phi'(x)=(x-\varepsilon)e^x+(\varepsilon-1)x+\varepsilon,\phi''(x)=(x-\varepsilon+1)e^x+\varepsilon-1,\\
\phi'''(x)=(x-\varepsilon+2)e^x\geq 0.
\end{gather*}
Since $\phi(0)=\phi'(0)=\phi''(0)=0$, we have $\phi\geq 0$. Moreover,
\begin{gather*}
(D-1)G(x)=(x-1)e^x-\frac{x^2}2+1,D(D-1)G(x)=x(xe^x-x),\\
(D-1)^2G(x)=(x^2-x+1)e^x-\frac{x^2}2-1,\\
[(D-1)^2-\varepsilon]G(x)=(x^2-x+1-\varepsilon)e^x+(\varepsilon-1)+(\varepsilon-1)\frac{x^2}2+\varepsilon x:=\psi(x),\\
\psi'(x)=(x^2+x-\varepsilon)e^x+(\varepsilon-1)x+\varepsilon,\psi''(x)=(x^2+3x-\varepsilon+1)e^x+\varepsilon-1,\\
\psi'''(x)=(x^2+5x-\varepsilon+4)e^x\geq 0.
\end{gather*}
Since $\psi(0)=\psi'(0)=\psi''(0)=0$, we have $\psi\geq 0$. This finishes the proof.
\end{enumerate}
\end{proof}
Results proved in this paper are listed in the following subsection.
\subsection{Main results}
The first result deals with well-posedness of \eqref{eq1}. We obtain existence of a unique global solution assuming that the nonlinearity satisfies \eqref{sc} or $[$\eqref{cc} with small data$]$. Let start with the subcritical case.
\begin{thm}\label{lwp'}
Assume that $G$ satisfies \eqref{sc} and take $u_0\in\Sigma$.  Then, there exist $T>0$ and a unique $u\in C_T(\Sigma)$ solution to the Cauchy problem \eqref{eq1}. Moreover,
\begin{enumerate}
\item
$u\in L^4_T(W^{1,4}(\R^2)),$
\item
$u$ satisfies conservation of the energy and the mass,
\item
in the defocusing case, $u$ is global.
\end{enumerate}
\end{thm}
Let treat the critical case with small data.
\begin{thm}\label{lwp}
Assume that $G$ satisfies \eqref{cc} and take $u_0\in\Sigma$ such that $\|\nabla u_0\|^2<\frac{4\pi}{\alpha_g}$. Then, there exist $T>0$ and a unique $u\in C_T(\Sigma)$ solution to the Cauchy problem \eqref{eq1}. Moreover,
\begin{enumerate}
\item
$u\in L^4_T(W^{1,4}(\R^2)),$
\item
$u$ satisfies conservation of the energy and the mass,
\item
in the defocusing case, $u$ is global if $E(u_0)\leq\frac{4\pi}{\alpha_g}$.
\end{enumerate}
\end{thm}
If we omit the condition of small data size in the critical case, we still have a local solution.
\begin{thm}\label{per}
Assume that $G$ satisfies \eqref{cc} and take $u_0\in\Sigma$. Then, there exist $T>0$ and at least a solution $u$ to the Cauchy problem \eqref{eq1} in the class $C_T(\Sigma)$. Moreover, $u\in L^4_T(W^{1,4}(\R^2))$ and satisfies conservation of the energy and the mass.
\end{thm}
Next, we are interested on the focusing Schr\"odinger problem \eqref{eq1}. This case is related to the associated stationary problem. Indeed under the condition \eqref{f} or \eqref{4}, we prove existence of a ground state $\phi$, in the meaning that 
$$\Delta\phi-\phi+|x|^2\phi+|x|^\mu G'(|\phi|^2)=0,\quad0\neq\phi\in \Sigma$$
and $\phi$ minimizes the problem
\begin{equation}\label{min}
m_{\alpha,\beta}:=\inf_{0\neq v\in{\Sigma}}\{S(v):=E(v)+ M(v) \quad\mbox{s. t}\quad K_{\alpha,\beta}(v)=0\},
\end{equation}
where, $\alpha,\beta\in\R$ and
$$\frac12K_{\alpha,\beta}(v):=\alpha\|\nabla v\|^2+(\alpha+\beta)\|v\|^2+(\alpha+2\beta)\|x\phi\|^2-\int|x|^\mu\Big[\alpha |v|g(|v|)+\beta(1+\frac\mu2)G(|v|^2)\Big]dx.$$
Precisely, we prove the result.
\begin{thm}\label{tgs}
Assume that $g$ satisfies \eqref{f} with $[$\eqref{sc} or \eqref{cc}$]$. Let two real numbers $(0,0)\neq(\alpha,\beta)\in\mathcal \R_+^2$. Thus, 
\begin{enumerate}
\item If $\alpha\beta\neq0$, then there is a minimizer of \eqref{min}, which is the energy of some solution to
\begin{equation}\label{gs}
\Delta \phi-\phi+|x|^2\phi+|x|^\mu g(\phi)=0,\quad 0\neq \phi\in \Sigma,\quad m_{\alpha,\beta}=S(\phi).
\end{equation}
\item
If $\beta=0$ and $\exists\varepsilon_g>0$ such that $(D-1-\varepsilon_g)G>0$ on $\R_+^*$. Then, there is a minimizer of \eqref{min}, which is the energy of some solution to \eqref{gs}. 
\item
If $\alpha=0$ then, there is a minimizer of \eqref{min}, which is the energy of some solution to the mass-modified equation
$$\Delta \phi-c\phi+d|x|^2\phi+|x|^\mu g(\phi)=0,\quad 0\neq \phi\in \Sigma,\quad m_{\alpha,\beta}=S(\phi).$$
\end{enumerate}
\end{thm}
The next result is about global existence or finite time blow-up of solution to the Schr\"odinger problem \eqref{eq1} with data in some stable sets. Here and hereafter, we denote for $\alpha,\beta\in\R$, the sets $A_{\alpha,\beta}^{+}:=\{v\in \Sigma\quad\mbox{s. t}\quad S(v)<m_{\alpha,\beta}\quad\mbox{and}\quad K_{\alpha,\beta}(v)\geq0\}$ and $ A_{\alpha,\beta}^{-}:=\{v\in \Sigma\quad\mbox{s. t}\quad S(v)<m_{\alpha,\beta}\quad\mbox{and}\quad K_{\alpha,\beta}(v)<0\}.$
\begin{thm}\label{sch}
Assume that $\epsilon=1$ and $g$ satisfies \eqref{4} with $[$\eqref{sc} or \eqref{cc}$]$. Let $u\in C_{T^*}(\Sigma)$ the maximal solution to \eqref{eq1}. Then,
\begin{enumerate}
\item
If there exist $(\alpha,\beta)\in\R_+^*\times\R_+\cup\{(1,-1)\}$ and $t_0\in[0,T^*)$ such that $u(t_0)\in A_{\alpha,\beta}^-$, then $u$ blows-up in finite time.
\item
If there exist $(\alpha,\beta)\in\R_+^*\times\R_+\cup\{(1,-1)\}$ and $t_0\in[0,T^*)$ such that $u(t_0)\in A_{\alpha,\beta}^+$, then $u$ is global.
\end{enumerate}
\end{thm}
The last result concerns nonlinear instability for the standing waves of the Schr\"odinger problem \eqref{eq1}.
\begin{defi}
Let $\phi$ a ground state to \eqref{gr}. The standing wave $e^{it}\phi$ is called orbitally stable if, for any $\varepsilon>0$, there exists $\sigma>0$ such that if
$ \inf_{\theta\in\R}\|u_0-e^{i\theta}\phi\|_\Sigma<\sigma$ then \eqref{eq1} has a global solution in $C(\R,\Sigma)$ satisfying $\sup_{t\in\R} \inf_{\theta\in\R}\|u(t)-e^{i\theta}\phi\|_\Sigma<\varepsilon$. Otherwise, the standing wave $e^{it}\phi$ is said to be nonlinearly instable.
\end{defi}
\begin{thm}\label{t2}
Let $\phi\in\Sigma$ a ground sate to \eqref{gs} such that
$$2\|\phi\|^2+\int|x|^\mu\Big[|\phi|^2g'(|\phi|)-(5+2\mu)|\phi|g(|\phi|)+(2+\mu)(1+\frac\mu2) G(|\phi|^2)\Big]\,dx>0.$$
Then the standing wave $e^{it}\phi$ is nonlinearly instable.
\end{thm}
We list in what follows some intermediate results.
\subsection{Tools}
This subsection is devoted to give some estimates needed along this paper. Let start with a classical tool to study Schr\"odinger problems which is the so-called Strichartz type estimate.
\begin{defi}
A pair $(q,r)$ of positive real numbers is said to be admissible if
$$2\leq r<\infty\quad\mbox{and}\quad \frac1q+\frac1r=\frac12.$$
\end{defi}
In order to control an eventual solution to \eqref{eq1}, we will use the following 
\begin{prop}\label{str}{\bf (Strichartz estimate \cite{Cas,rm})}
For any $T>0$, any admissible pairs $(q,r)$ and $(\alpha,\beta)$,
$$\|u\|_{L^q_T(L^r)}\leq C_{\alpha,T}\Big(\|u_0\|+\|i\dot u-|x|^2u+\Delta u\|_{L^{\alpha'}_T(L^{\beta'})}\Big).$$
\end{prop}
In order to estimate the quantity $\int G(|u|^2)\,dx$ which is a part of the energy, we will use Moser-Trudinger type inequalities \cite{Ad,Mo,Tr}.
\begin{prop}\label{prop3}{\bf(Moser-Trudinger inequality)}
Let $\alpha\in (0,4\pi)$, a constant $\mathcal{C}_{\alpha}$ exists such that
for all $u\in H^{1}$ satisfying $\|\nabla u\|\leq 1$, we have
$$\int\Big({\rm e}^{\alpha |u(x)|^{2}}-1\Big)dx\leq \mathcal{C}_{\alpha}\|u\|^{2}.$$
Moreover, this inequality is false if $\alpha\geq 4\pi$ and
$\alpha=4\pi$ becomes admissible if we take $\|u\|_{H^{1}}\leq 1$ rather than
$\|\nabla u\|\leq 1$. In this case
$$\mathcal K:=\displaystyle\sup_{\|u\|_{H^{1}}\leq 1}\int\Big({\rm e}^{4\pi |u(x)|^{2}}-1\Big)dx<\infty$$
and this is false for $\alpha>4\pi$. See \cite{Ru} for more details.
\end{prop}
Despite the lack of injection of $H^1$ on the bounded functions set, we can control the $L^\infty$ norm by the $H^1$ norm and some H\"older norm with a logarithmic growth.
\begin{prop}\label{log}{\bf(Log estimate \cite{Ib1})}
Let $\beta\in ]0,1[$. For any $\lambda>\frac{1}{2\pi\beta}$, any $0<\omega\leq1$, a constant $C_{\lambda}$ exists such that, for any function $u\in (H^1\cap C^{\beta})(\R^2)$, we have
$$\|u\|_{L^{\infty}}^2\leq\lambda\|u\|_{\omega}^2\log\Big(C_{\lambda}+\frac{8^{\beta}\|u\|_{C^{\beta}}}{\omega^{\beta}\|u\|_{\omega}}\Big),$$
where $\|u\|_{\omega}^2:=\|\nabla u\|_{L^2(\R^2)}^2+\omega^2\|u\|_{L^2}^2.$
\end{prop}
In the next section, we will use the $L^{\infty}$ logarithmic estimate for $\beta=\frac12$, coupled with the continuous Sobolev injection $W^{1,4}(\R^2)\hookrightarrow C^\frac12(\R^2)$.
The following absorption result will be useful.
\begin{lem}\label{abs}{\bf{(Bootstrap Lemma  \cite{taobk})}}
Let $T>0$ and $X\in C([0,T],\R_+)$ such that $$X\leq a+bX^{\theta},\mbox{ on } [0,T],$$
where, $a,b>0,\theta>1,a<(1-\frac{1}{\theta})\frac{1}{(\theta b)^{\frac{1}{\theta}}}$ and $X(0)\leq \frac{1}{(\theta b)^{\frac{1}{\theta-1}}}$. Then
$$X\leq\frac{\theta}{\theta -1}a, \mbox{ on } [0,T].$$
\end{lem}
We close this section with the following result.
\begin{prop}\label{prp2''}{\bf (Sobolev embedding \cite{AC1,jc})}\label{sblv}
In two space dimensions
\begin{enumerate}
\item
$W^{s,p}\hookrightarrow L^q$, for any $1<p<q<\infty,s>0$ such that $\frac{1}{p}\leq\frac{1}{q}+\frac{s}{2}$.
\item
$|u(x)|\lesssim \frac{\|u\|_{H^1}}{\sqrt{|x|}}$, for all $u\in H^1_{rd}$ and almost all $|x|>0$.
\item
$\||x|^bu^p\|_1\lesssim \|\nabla u\|^{p-2-b}\|u\|^{2+b}$ for all $u\in H^1_{rd}$, $b\geq0$ and $p>2+2b$.
\end{enumerate}
\end{prop}
\section{Well-posedness in  the subcritical case}
This section is devoted to prove Theorem \ref{lwp'} about global well-posedness of the nonlinear Schr\"odinger problem \eqref{eq1} in the subcritical case. So, we assume in all this section that \eqref{sc} is satisfied.\\ Let us identify $\C$ with $\R^2$ and $g$ with a function defined on $\R^2$. Denote by $\mathcal D g$ the $\R^2$ derivative of the identified function. Then using \eqref{sc}, the
mean value theorem and the convexity of the exponential function, we derive the
following  property
\begin{lem}\label{g'}
There exists $q:=q_g>2+2\mu$ such that for any $\varepsilon >0$, there exists $C_\varepsilon>0$ satisfying
\begin{gather*}
|g(z_1)-g(z_2)|\leq C_{\varepsilon} |z_1-z_2|\sum_{i=1}^2|z_i|^{q-1}e^{ \varepsilon|z_i|^2},\quad\forall z_1,z_2\in\C,\\
|\mathcal Dg(z_1)-\mathcal Dg(z_2)|\leq C_{\varepsilon}|z_1-z_2|\sum_{i=1}^2|z_i|^{q-2}{e}^{ \varepsilon|z_i|^2},\quad\forall z_1,z_2\in\C.
\end{gather*}
\end{lem}
The proof of Theorem \ref{lwp'} contains three steps. First we prove existence of a local solution, second we show uniqueness and third we obtain global well-posedness. In the two next subsections, we assume that $\epsilon=1$. Indeed, the sign of $\epsilon$ has no local effect.
\subsection{Local well-posedness}
We use a standard fixed point argument. For $T>0$, denote the space
$$X_T:=\{u\in C_T(\Sigma)\quad\mbox{s. t}\quad u,\nabla u,xu\in L^4_T(L^4)\}$$
endowed with the complete norm
$$\|u\|_T:=\|u\|_{L^{\infty}_T(H^1)}+\|u\|_{L^4_T(W^{1,4})}+\|xu\|_{L^{\infty}_T(L^2)}+\|xu\|_{L^4_T(L^4)}.$$
Let the map $$\phi:\quad v\longmapsto -i\int_0^t U(t-s)[|x|^\mu g(v+w)(s)]ds,$$
where $w:=U(t)u_0$ is the solution to the associated free problem to \eqref{eq1}, namely, for $V:=|x|^2$,
$$i\dot w+\Delta w=|x|^2w\quad w(0,.)=u_0.$$
We shall prove that $\phi$ is a contraction on the closed unit ball of $X_T$ for some positive time $T>0$. Using Strichartz estimate in Proposition \ref{str} with the fact that
$\nabla \phi(v)=-i\int_0^t U(t-s)[\nabla(|x|^\mu g(v+w)(s))ds+\nabla V\phi(v)]ds$ and
$x\phi(v)=-i\int_0^t U(t-s)[x |x|^\mu g(v+w)(s)ds+2\nabla\phi(v)]ds$, we have
$$\|\phi(v)\|_{L^{\infty}_T(L^2)\cap L^4_T(L^4)}\lesssim \||x|^\mu g(v+w)\|_{L^1_T(L^2)},$$
\begin{eqnarray*}
\|\nabla(\phi(v))\|_{L^{\infty}_T(L^2)\cap L^4_T(L^4)}
&\lesssim& \|\nabla(|x|^\mu g(v+w))\|_{L^1_T(L^2)}+\|\phi(v)\nabla V\|_{L^1_T(L^2)}\\
&\lesssim& \|\nabla(|x|^\mu g(v+w))\|_{L^1_T(L^2)}+T\|x\phi(v)\|_{L^{\infty}_T(L^2)},
\end{eqnarray*}
\begin{eqnarray*}
\|x\phi(v)\|_{L^{\infty}_T(L^2)\cap L^4_T(L^4)}
&\lesssim& \|x|x|^\mu g(v+w)\|_{L^1_T(L^2)}+T\|\nabla (\phi(v))\|_{L^{\infty}_T(L^2)}.
\end{eqnarray*}
Thus
\begin{equation}
\|\phi(v)\|_{T}\lesssim \||x|^\mu g(v+w)\|_{L^1_T(\Sigma)}+T\Big(\|\nabla (\phi(v))\|_{L^{\infty}_T(L^2)}+\|x\phi(v)\|_{L^{\infty}_T(L^2)}\Big).
\end{equation}
Now, let $v\in B_T(1)$ the closed unit ball of $X_T$. By \eqref{sc}, for any $\varepsilon>0$ there exists $C_\varepsilon>0$ such that
\begin{eqnarray*}
\|x|x|^\mu g(v+w)\|
&\leq&C_\varepsilon\||x|^{1+\mu}|v+w|^qe^{\varepsilon|v+w|^2}\|\\
&\lesssim&\||x|^{1+\mu}|v+w|^q(e^{\varepsilon|v+w|^2}-1)\|+\||x|^{1+\mu}(v+w)^q\|\\
&\lesssim&\||x|(v+w)\|_4[\||x|^{\mu}(v+w)^{q-1}\|_8\|e^{\varepsilon|v+w|^2}-1\|_8+\||x|^{\mu}|v+w|^{q-1}\|_4].
\end{eqnarray*}
On the other hand, by the conservation of the mass and the energy of $w$, 
$$\|v+w\|_{H^1}^2\leq 2(\|v\|_{H^1}^2+\|w\|_{H^1}^2)\leq2(1+ \|u_0\|_{\Sigma}^2).$$ Take
\begin{equation}\label{eps}
\varepsilon_0:=\frac\pi{4(1+\|u_0\|_\Sigma)^2}.
\end{equation}
Using Moser-Trudinger inequality, we have
\begin{eqnarray*}
\|e^{\varepsilon_0|v+w|^2}-1\|_{L^8}^8
&\lesssim&\int\Big(e^{8\varepsilon_0\|v+w\|_{H^1}^2(\frac{|v+w|}{\|v+w\|_{H^1}})^2}-1\Big)dx\\
&\lesssim& \|v+w\|^2\lesssim \Big(1+\|u_0\|\Big)^2.
\end{eqnarray*}
By the interpolation inequality in proposition \eqref{prp2''}, since $4(q-1)>2+8\mu$ and $8(q-1)>2+16\mu$,
$$\||x|^{\mu}(v+w)^{q-1}\|_4+\||x|^{\mu}(v+w)^{q-1}\|_8\lesssim\|v+w\|_T^{q-1}.$$
Thus
\begin{eqnarray*}
\|x|x|^\mu g(v+w)\|_{L^1_T(L^2)}
&\lesssim&\|x(v+w)\|_{L_T^1(L^4)}[(1+\|u_0\|)^{\frac14}+1]\|v+w\|_T^{q-1}T^\frac34\\
&\lesssim&T^\frac34(1+\|u_0\|)^{\frac14}\|v+w\|_{T}^q\\
&\lesssim&T^\frac34(1+\|u_0\|)^{\frac{1}{4}}(1+\|u_0\|_\Sigma)^q\\
&\lesssim&(1+\|u_0\|_\Sigma)^{q+\frac14}T^\frac34.
\end{eqnarray*}
It remains to control 
$$\||x|^\mu g(v+w)\|_{L^1_T(\dot H^1)}\lesssim\||x|^\mu\nabla[ g(v+w)]\|_{L_T^1(L^2)}+\||x|^{\mu-1} g(v+w)\|_{L_T^1(L^2)}.$$ 
By Lemma \ref{g'}, since $q>2+2\mu$, arguing as previously
\begin{eqnarray*}
\||x|^{\mu-1} g(v+w)]\|
&\lesssim&\||x|^{\mu-1}|v+w|^{q}(e^{\varepsilon_0|v+w|^2}-1)\|+\||x|^{\mu-1}|v+w|^{q}\|\\
&\lesssim&\||x|^{\mu-1}|v+w|^{q}\|_4\|e^{4\varepsilon_0|v+w|^2}-1\|_1^\frac14+\||x|^{\mu-1}|v+w|^{q}\|\\
&\lesssim&\|v+w\|_T^{q}[1+(1+\|u_0\|)^\frac12].
\end{eqnarray*}
Moreover, also by Lemma \ref{g'}, via H\"older inequality
\begin{eqnarray*}
\||x|^\mu\nabla[ g(v+w)]\|
&\lesssim&\||x|^\mu\nabla(v+w)|v+w|^{q-1}e^{\frac{\varepsilon_0}2|v+w|^2}\|\\
&\lesssim&\|x\nabla(v+w)\|_4\||x|^{\mu-1}|v+w|^{q-1}e^{\frac{\varepsilon_0}2|v+w|^2}\|_4\\
&\lesssim&\|x\nabla(v+w)\|_4[\||x|^{\mu-1}|v+w|^{q-1}(e^{\frac{\varepsilon_0}2|v+w|^2}-1)\|_4+\||x|^{\mu-1}|v+w|^{q-1}\|_4]\\
&\lesssim&\|x\nabla(v+w)\|_4[\||x|^{\mu-1}|v+w|^{q-1}\|_8\|e^{4\varepsilon_0|v+w|^2}-1\|_1^{\frac18}+\||x|^{\mu-1}|v+w|^{q-1}\|_4]\\
&\lesssim&\|x\nabla(v+w)\|_4\|v+w\|_T^{q-1}[1+(1+\|u_0\|)^\frac14].
\end{eqnarray*}
This implies that
$$\||x|^\mu g(v+w)\|_{L^1_T(\dot H^1)}\lesssim T\|v+w\|_T^{q}(1+\|u_0\|)^\frac12+ T^\frac34\|v+w\|_T^q(1+\|u_0\|)^\frac14.$$
Therefore, for $0<T<1$ small enough,
\begin{eqnarray*}
\|\phi(v)\|_{T}
&\lesssim &\|g(v+w)\|_{L^1_T(\Sigma)}+T\|\phi(v)\|_{T}\\
&\lesssim&\Big(1+\|u_0\|_{\Sigma}\Big)^{q+\frac1{2}}T^{\frac{3}{4}}+T\|\phi(v)\|_{T}\\
&\lesssim&\Big(1+\|u_0\|_{\Sigma}\Big)^{q+\frac1{2}}\frac{T^{\frac{3}{4}}}{1-T}.
\end{eqnarray*}
Thus, for $T>0$ small enough, $\phi$ maps $B_T(1)$ into itself. It remains to prove that $\phi$ is a contraction. Let $v_1,v_2\in B_T(1)$ solutions to \eqref{eq1}. Then
$$\phi(v_1)-\phi(v_2)=-i\int_0^t U(t-s)\Big(|x|^\mu[g(v_1+w)-g(v_2+w)(s)]\Big)ds.$$
Using Strichartz estimate in Proposition \ref{str} and arguing as previously, we have{\small
\begin{gather*}
\|\phi(v_1)-\phi(v_2)\|_{L^{\infty}_T(L^2)\cap L^4_T(L^4)}\lesssim \||x|^\mu[g(v_1+w)-g(v_2+w)]\|_{L^1_T(L^2)},\\
\|\nabla(\phi(v_1)-\phi(v_2))\|_{L^{\infty}_T(L^2)\cap L^4_T(L^4)}
\lesssim \|\nabla(|x|^\mu[g(v_1+w)-g(v_2+w)])\|_{L^1_T(L^2)}+T\|x(\phi(v_1)-\phi(v_2))\|_{L^{\infty}_T(L^2)},\\
\|x(\phi(v_1)-\phi(v_2))\|_{L^{\infty}_T(L^2)\cap L^4_T(L^4)}
\lesssim \|x|x|^\mu(g(v_1+w)-g(v_2+w))\|_{L^1_T(L^2)}+T\|\nabla (\phi(v_1)-\phi(v_2))\|_{L^{\infty}_T(L^2)}.
\end{gather*}}
Thus, for small $T>0$,
\begin{eqnarray*}
\|\phi(v_1)-\phi(v_2)\|_{T}
&\lesssim& \|g(v_1+w)-g(v_2+w)\|_{L^1_T(\Sigma)}+T\|\phi(v_1)-\phi(v_2)\|_{L^{\infty}_T(\Sigma)}\\
&\lesssim&\frac{1}{1-T}\|g(v_1+w)-g(v_2+w)\|_{L^1_T(\Sigma)}.
\end{eqnarray*}
Denoting $u_i=v_i+w$ and $u:=v_1-v_2$, using Moser-Trudinger inequality via Lemma \ref{g'}, since $q-1>2+2(\mu-1)$,
\begin{eqnarray*}
\|x|x|^\mu(g(u_1)-g(u_2))\|
&\lesssim&\sum_{i=1}^2\||x|^{\mu+1}u|u_i|^{q-1}(e^{\varepsilon_0|u_i|^2}-1)\|+\||x|^{\mu+1}u|u_i|^{q-1}\|\\
&\leq&\|u\|_4[\||x|^{\mu+1}|u_1|^{q-1}\|_{L^8}\|e^{8\varepsilon_0|u_1|^2}-1\|_1^{\frac18}+\||x|^{\mu+1}|u_i|^{q-1}\|_4]\\
&\lesssim&\|u\|_T\|u_1\|_T^{q-1}(1+\|u_0\|)^{\frac14}\\
&\lesssim&\|u\|_T(1+\|u_0\|_\Sigma)^{q-\frac34}.
\end{eqnarray*}
So
\begin{eqnarray*}
\|x|x|^\mu(g(u_1)-g(u_2))\|_{L^1_T(L^2)}
&\lesssim&(1+\|u_0\|_\Sigma)^{q-\frac34}\|u\|_{L^\infty_T(H^1)}T\\
&\lesssim&(1+\|u_0\|_\Sigma)^{q-\frac34}\|v_1-v_2\|_TT.
\end{eqnarray*}
Compute, fore $|x|^\mu(g(u_1)-g(u_2)):=h$,
{\small \begin{eqnarray*}
\|\nabla h\|_{L^1_T(L^2)}
&=&\| |x|^\mu(\mathcal Dg(u_1)-\mathcal Dg(u_2))\nabla u_1+|x|^\mu\mathcal Dg(u_2)\nabla u+\mu|x|^{\mu-1}(g(u_1)-g(u_2))\|_{L^1_T(L^2)}\\
&\leq&\|\nabla u_1|x|^\mu(\mathcal Dg(u_1)-\mathcal Dg(u_2))\|_{L^1_T(L^2)}+\|\nabla u|x|^\mu\mathcal Dg(u_2)\|_{L^1_T(L^2)}+\mu\||x|^{\mu-1}(g(u_1)-g(u_2))\|_{L^1_T(L^2)}\\
&:=& (A)+(B)+(C).
\end{eqnarray*}}
By H\"older and Moser-Trudinger inequalities via Lemmas \ref{sblv} and \eqref{g'},
\begin{eqnarray*}
(B)
&\lesssim&\|\nabla u|x|^\mu u_2^{q-1}(e^{\frac{\varepsilon_0}8|u_2|^2}-1)\|_{L^1_T(L^2)}+\|\nabla u|x|^\mu u_2^{q-1}\|_{L^1_T(L^2)}\\
&\lesssim& T^\frac34\|\nabla u\|_{L_T^4(L^4)}\Big[\||x|^{\mu} u_2^{q-1}\|_{L_T^\infty(L^8)}\|e^{\frac{\varepsilon_0}8|u_2|^2}-1\|_{L^\infty_T(L^{8})}+\||x|^{\mu} u_2^{q-1}\|_{L_T^\infty(L^4)}\Big]\\
&\lesssim& T^\frac34(1+\|u_0\|)^\frac14\|u\|_{T}\|u_2\|_T^{q-1}.
\end{eqnarray*}
Let estimate $(C)$. Write using \eqref{g'} via Sobolev, H\"older and Moser-Trudinger inequalities with previous calculations,
\begin{eqnarray*}
(C)
&\lesssim&\|u|x|^{\mu-1}u_1^{q-1}(e^{\frac{\varepsilon_0}8|u_1|^2}-1)\|_{L^1_T(L^2)}+\|u|x|^{\mu-1}u_1^{q-1}\|_{L^1_T(L^2)}\\
&\lesssim&\|u|x|^{\mu-1}u_1^{q-1}\|_{L^1_T(L^4)}\|e^{\frac{\varepsilon_0}8|u_1|^2}-1\|_{L^\infty_T(L^4)}+\|u|x|^{\mu-1}u_1^{q-1}\|_{L^1_T(L^2)}\\
&\lesssim&T\|u\|_{L_T^\infty(H^1)}\||x|^{\mu-1}u_1^{q-1}\|_{L^\infty_T(L^8)}(1+\|u_0\|)^\frac12+T\|u\|_{L_T^\infty(H^1)}\||x|^{\mu-1}u_1^{q-1}\|_{L^1_\infty(L^4)}\\
&\lesssim&(1+\|u_0\|)^\frac12T\|u\|_T\|u_1\|_T^{q-1}.
\end{eqnarray*}
Let control $(A)$. Taking $p:=4+\frac2{q-2-2\mu}$ in Lemma \ref{g'}, yields $q>2\mu+2+\frac2p$ and via Lemma \ref{sblv},
\begin{eqnarray*}
(A)
&\lesssim&\sum_{i=1}^2\|\nabla u_1|x|^{\mu}uu_i^{q-2}e^{(\frac14-\frac1p){\varepsilon_0}|u_i|^2}\|_{L^1_T(L^2)}\\
&\lesssim&T^\frac34\|\nabla u_1\|_{L^4_T(L^4)}\||x|^{\mu}uu_1^{q-2}e^{(\frac14-\frac1p){\varepsilon_0}|u_1|^2}\|_{L^\infty_T(L^{4})}\\
&\lesssim&T^\frac34\|u_1\|_{T}\|u\|_{L_T^\infty(H^1)}\||x|^{\mu}u_1^{q-2}\|_{L_T^\infty(L^p)}[\|e^{(\frac14-\frac1p)\varepsilon_0|u_1|^2}-1\|_{L^\infty_T(L^{\frac{8p}{p-4}})}+1]\\
&\lesssim&T^\frac34\|u\|_T\|u_1\|_T^{q-1}[\|e^{2\varepsilon_0|u_1|^2}-1\|_{L^\infty_T(L^1)}^{\frac18-\frac1{2p}}+1]\\
&\lesssim&T^\frac34(1+\|u_0\|)^\frac14\|u_1\|_T^{q-1}\|u\|_T.
\end{eqnarray*}
Finally,
$$\|\phi(v_1)-\phi(v_2)\|_T\leq C[T^{\frac34}+T](1+\|u_0\|_\Sigma)^{q}\|v_1-v_2\|_T.$$
So $\phi$ is a contraction of $B_T(1)$ for some $T>0$ small enough. It's fix point $v$ satisfies $u=v+w$ is a solution to \eqref{eq1}. The existence is proved.
\subsection{Uniqueness in the conformal space}
We prove uniqueness of solution to \eqref{eq1} in the conformal space. Let $u_1,u_2\in C_T(\Sigma)$ solutions to \eqref{eq1} and $u:=u_1-u_2$. Compute
$$i\dot u+\Delta u=|x|^2u+|x|^\mu(g(u_1)-g(u_2)),\quad u(0,.)=0.$$
With a continuity argument we may assume that $0<T<1$ and
\begin{equation}\label{cnt}
\max_{i\in\{1,2\}}\|u_i\|_{L^\infty_T(\Sigma)}\leq 1+\|u_0\|_\Sigma.
\end{equation}
Uniqueness follows from previous computation with the result
\begin{lem}\mbox{}\\
\begin{enumerate}
\item
$\|xu\|_{L^{\infty}_T(L^2)\cap L^4_T(L^4)}\lesssim\frac{T}{1-(1+\|u_0\|_{\Sigma})^{q-\frac34}T^{\frac{3}{4}}}\|u\|_{L^{\infty}_T(H^1)}.$
\item
$\|\nabla u_1\|_{L^4_T(L^4)}\lesssim\frac{\|u_0\|+(1+\|u_0\|_{\Sigma})^{q+\frac14}T}{1-T^\frac34(1+\|u_0\|_\Sigma)^{q-\frac34}}.$
\item
$\|\nabla u\|_{L^4_T(L^4)}\lesssim\frac{T\|x u\|_{L^{\infty}_T(L^2)}+(1+\|u_0\|_\Sigma)^{q-\frac12}(1+\|\nabla u_1\|_{L_T^4(L^4)})T^\frac34\|u\|_{L_T^\infty(H^1)}}{1-T^\frac34(1+\|u_0\|_\Sigma)^{q-\frac12}}.$
\end{enumerate}
\end{lem}
\begin{proof}
Using Moser-Trudinger inequality via Lemma \ref{g'}, we have
\begin{eqnarray*}
\|xu\|_{L^{\infty}_T(L^2)\cap L^4_T(L^4)}
&\lesssim& \|x|x|^\mu(g(u_1)-g(u_2))\|_{L^1_T(L^2)}+T\|\nabla u\|_{L^{\infty}_T(L^2)}\\
&\lesssim&\sum_i (\|xu|x|^\mu u_i^{q-1}(e^{\varepsilon_0|u_i|^2}-1)\|_{L^1_T(L^2)}+\|xu|x|^\mu u_i^{q-1}\|_{L_T^1(L^2)})+T\|u\|_{L^{\infty}_T(H^1)}\\
&\lesssim& T^{\frac{3}{4}}\|xu\|_{L^4_T(L^4)}\sum_i\|u_i\|_{L_T^\infty(H^1)}^{q-1}(1+\|e^{\varepsilon_0|u_i|^2}-1\|_{L^{\infty}_T(L^8)})+T\|u\|_{L^{\infty}_T(H^1)}\\
&\lesssim& \|xu\|_{L^4_T(L^4)}(1+\|u_0\|_{\Sigma})^{q-\frac34}T^{\frac{3}{4}}+T\|u\|_{L^{\infty}_T(H^1)}\\
&\lesssim&\frac{T}{1-(1+\|u_0\|_{\Sigma})^{q-\frac34}T^{\frac{3}{4}}}\|u\|_{L^{\infty}_T(H^1)}.
\end{eqnarray*}
Using Strichartz estimate, Moser-Trudinger inequality via Lemma \ref{g'} and arguing as previously
\begin{eqnarray*}
\|\nabla u_1\|_{L^4_T(L^4)}
&\lesssim&\|u_0\|+ \|\nabla[|x|^\mu g(u_1)]\|_{L^1_T(L^2)}+T\|x u_1\|_{L^{\infty}_T(L^2)}\\
&\lesssim&\|u_0\|+ \||x|^{\mu-1}g(u_1)\|_{L^1_T(L^2)}+\||x|^{\mu}\mathcal Dg(u_1)\nabla u_1\|_{L^1_T(L^2)}+T\|x u_1\|_{L^{\infty}_T(L^2)}\\
&\lesssim&\|u_0\|+ \||x|^{\mu-1}u_1^qe^{\varepsilon_0|u_1|^2}\|_{L^1_T(L^2)}+\||x|^{\mu}\nabla u_1u_1^{q-1}e^{\varepsilon_0|u_1|^2}\|_{L^1_T(L^2)}+T\|x u_1\|_{L^{\infty}_T(L^2)}\\
&\lesssim&\|u_0\|+(1+\|u_0\|_\Sigma)^{\frac14}[\|u_1\|_{L_T^\infty(H^1)}^qT+T^\frac34\|\nabla u_1\|_{L_T^4(L^4)}\|u_1\|_{L_T^\infty(H^1)}^{q-1}]+T(1+\|u_0\|_\Sigma)\\
&\lesssim& \frac{\|u_0\|+(1+\|u_0\|_{\Sigma})^{q+\frac14}T}{1-T^\frac34(1+\|u_0\|_\Sigma)^{q-\frac34}}.
\end{eqnarray*}
By Strichartz estimate, we have
\begin{eqnarray*}
\|\nabla u\|_{L^4_T(L^4)}
&\lesssim& \|\nabla[|x|^\mu(g(u_1)-g(u_2))]\|_{L^1_T(L^2)}+T\|x u\|_{L^{\infty}_T(L^2)}\\
&\lesssim& \||x|^{\mu-1}(g(u_1)-g(u_2))]\|_{L^1_T(L^2)}+\||x|^{\mu}\nabla(g(u_1)-g(u_2))\|_{L^1_T(L^2)}+T\|x u\|_{L^{\infty}_T(L^2)}.
\end{eqnarray*}
With previous computation, for $T\in(0,1)$ small enough
{\small\begin{eqnarray*}
\|\nabla u\|_{L^4_T(L^4)}
&\lesssim&(1+\|u_0\|_\Sigma)^\frac12T^\frac34[\|\nabla u\|_{L^4_T(L^4)}\|u_1\|_{L_T^\infty(H^1)}^{q-1}\\&+&\|u\|_{L_T^\infty(H^1)}\|u_1\|_{L_T^\infty(H^1)}^{q-2}(\|u_1\|_{L_T^\infty(H^1)}+\|\nabla u_1\|_{L_T^4(L^4)})]+T\|x u\|_{L^{\infty}_T(L^2)}\\
&\lesssim&(1+\|u_0\|_\Sigma)^{q-\frac12}T^\frac34[\|\nabla u\|_{L^4_T(L^4)}+\|u\|_{L_T^\infty(H^1)}(1+\|\nabla u_1\|_{L_T^4(L^4)})]+T\|x u\|_{L^{\infty}_T(L^2)}\\
&\lesssim&\frac{T\|x u\|_{L^{\infty}_T(L^2)}+(1+\|u_0\|_\Sigma)^{q-\frac12}(1+\|\nabla u_1\|_{L_T^4(L^4)})T^\frac34\|u\|_{L_T^\infty(H^1)}}{1-T^\frac34(1+\|u_0\|_\Sigma)^{q-\frac12}}.
\end{eqnarray*}}
\end{proof}
\subsection{Global well-posedness in the defocusing case}
This subsection is devoted to prove that the maximal solution of \eqref{eq1} is global in the defocusing case. Recall an important fact that is the time of local existence depends only on
the quantity $\|u_0\|_\Sigma$. Let $u$ to be the unique maximal
solution of \eqref{eq1} in the space $C_T(\Sigma)$ for any $0<T<{T^*}$ with initial data $u_0$,
where $0<T^*\leq +\infty$ is the lifespan of $u$. We shall prove that $u$ is global.
By contradiction, suppose that $T^*< +\infty$. Consider for $0<s<T^*$, the problem
$$(P_{s})\left\{\begin{array}{cccc}
i\dot v+\Delta{v}-|x|^2v&=&|x|^\mu g(v),\\
v(s,.)&=&u(s,.).
 \end{array}\right.$$
 Using the same arguments used in the local existence, we can find a real $\tau>0$ and a solution
$v$ to $(P_{s})$ on $[s,s+\tau]$. Taking in the section of local existence, instead of $\varepsilon_0$, the real number
 $$\varepsilon=\frac\pi{8(1+E(0)+M(0))},$$
we see that $\tau$ does not depend on $s$. Thus, if we let $s$ be
close to $T^*$ such that $s+\tau>T^*$, we can extend $v$ for times higher than $T^*$. This fact
contradicts the maximality of $T^*$. We obtain the result claimed.
\section{Well-posedness in the critical case}
This section is devoted to prove Theorem \ref{lwp} about existence of a unique solution to the nonlinear Schr\"odinger problem \eqref{eq1} in the critical case. So, we suppose in all this section that \eqref{cc} is satisfied. In this section we assume, for simplicity and without loss of generality that $\alpha_g=1$.
 Then using, \eqref{cc}, the
mean value theorem and the convexity of the exponential function, we derive the
following  property
\begin{lem}\label{g}
There exists $q>2+2\mu$ such that for any $\varepsilon >0$, there exists $C_\varepsilon>0$ satisfying
\begin{gather*}
|g(z_1)-g(z_2)|\leq C_{\varepsilon} |z_1-z_2|\sum_{i=1}^2|z_i|^{q-1}e^{ (1+\varepsilon)|z_i|^2},\quad\forall z_1,z_2\in\C,\\
|\mathcal Dg(z_1)-\mathcal Dg(z_2)|\leq C_{\varepsilon}|z_1-z_2|\sum_{i=1}^2|z_i|^{q-2}e^{ (1+\varepsilon)|z_i|^2},\quad\forall z_1,z_2\in\C.
\end{gather*}
\end{lem}
The next auxiliary result will be useful.
\begin{lem}\label{exp}
Let $u\in C_{T}(H^1_{rd})\cap L_T^4(W^{1,4})$ a solution to \eqref{eq1} satisfying $\|\nabla u\|_{L^\infty_T(L^2)}^2<4\pi$, then there exists two real number $\alpha<4$ near four and $\varepsilon>0$ near zero such that for any H\"older couple $(p,p')$,
$$\|e^{(1+\varepsilon)|u|^2}-1\|_{L^{p'}_T(L^p)}\lesssim T^{1-\frac1p}+\|u\|_{L_T^4(W^{1,4})}^\alpha T^{(1-\frac1p)(1-\frac\alpha4)}.$$
\end{lem}
\begin{proof}
By H\"older inequality, for any $\varepsilon>0$,
\begin{eqnarray*}
\|e^{(1+\varepsilon)|u|^2}-1\|_{L^{p'}_T(L^p)}
&\lesssim&\|e^{\frac1{p'}(1+\varepsilon)\|u\|_{L_x^\infty}^2}\|_{L^{p'}(0,T)}\|e^{(1+\varepsilon)|u|^2}-1\|_{L^{\infty}_T(L^1)}^\frac1p.
\end{eqnarray*}
 We can find $\varepsilon>0$ small such that $(1+\varepsilon)\|\nabla u\|^2<4\pi$. So, by Moser-Trudinger inequality,
\begin{eqnarray*}
\int\Big(e^{(1+\varepsilon)|u|^2}-1\Big)\,dx
&\leq&\int\Big(e^{(1+\varepsilon)\|\nabla u\|^2(\frac{|u|}{\|\nabla u\|})^2}-1\Big)\,dx\lesssim\|u\|^2\lesssim 1.
\end{eqnarray*}
For any $\lambda>\frac{1}{\pi}$ and $\omega\in ]0,1]$, by the Logarithmic inequality in Proposition \ref{log}, we have
$$e^{(1+\varepsilon)\|u\|^2_{L^{\infty}_x}}\leq\Big(C+2\sqrt{\frac{2}{\omega}}\frac{\|u\|_{C^{\frac{1}{2}}}}{\|u\|_{\omega}}\Big)^{\lambda(1+\varepsilon)\|u\|_{\omega}^2}.$$
Since $\|u\|_{\omega}^2=\omega^2\|u\|^2+\|\nabla u\|^2$, we may take $0<\omega,\varepsilon$ near zero and $\alpha<4$ near four such that $(1+\varepsilon)\|u\|_{\omega}^2<\alpha\pi<4\pi$. Thus, for $\lambda$ near $\frac1\pi$, we have
\begin{eqnarray*}
e^{(1+\varepsilon)\|u\|^2_{L^{\infty}_x}}
&\leq&\Big(C+2\sqrt{\frac{2}{\omega}}\frac{\|u\|_{C^{\frac{1}{2}}}}{\|u\|_{\omega}}\Big)^{\lambda(1+\varepsilon)\|u\|_{\omega}^2}\nonumber\\
&\lesssim&\Big(1+\|u\|_{C^{\frac{1}{2}}}\Big)^\alpha\lesssim1+\|u\|_{W^{1,4}}^\alpha.
\end{eqnarray*}
It follows that
\begin{eqnarray*}
\|e^{(1+\varepsilon)|u|^2}-1\|_{L^{p'}_T(L^p)}
&\lesssim&\|e^{\frac1{p'}(1+\varepsilon)\|u\|_{L_x^\infty}^2}\|_{L^{p'}(0,T)}\|e^{(1+\varepsilon)|u|^2}-1\|_{L^{\infty}_T(L^1)}^\frac1p\\
&\lesssim&\|e^{\frac1{p'}(1+\varepsilon)\|u\|_{L_x^\infty}^2}\|_{L^{p'}(0,T)}\\
&\lesssim&\|1+\|u\|_{W^{1,4}}^\alpha\|_{L^1(0,T)}^\frac1{p'}\\
&\lesssim&T^{1-\frac1p}+\|u\|_{L_T^4(W^{1,4})}^\alpha T^{(1-\frac1p)(1-\frac\alpha4)}.
\end{eqnarray*}
\end{proof}
The proof of Theorem \ref{lwp} contains three steps. First we prove existence of a local solution, second we show uniqueness and third we obtain global well-posedness. In the two next subsections, we assume that $\epsilon=1$. Indeed, the sign of $\epsilon$ has no local effect.
\subsection{Local well-posedness}
We use a standard fixed point argument. For $T>0$, we keep notations of the previous section.
We shall prove the existence a small $T>0$ such that $\phi$ is a contraction on some closed ball of $X_T$. Let $v\in B_T(r)$ the closed ball of $X_T$ centered on zero and with radius $r>0$. Using Strichartz estimate, we have, for $T\in(0,1)$,
$$\|\phi(v)\|_{T}\lesssim \||x|^\mu g(v+w)\|_{L^1_T(\Sigma)}.$$
 Taking account of Lemma \ref{g} via H\"older inequality and the estimate on $\R_+$,  $r^qe^{(1+\varepsilon)r^2}\leq C_{q,\varepsilon}(r^q+r^{q+2}e^{(1+\varepsilon)r^2})$, for any $\varepsilon>0$,
\begin{eqnarray*}
\|x|x|^\mu g(v+w)\|_{L^1_T(L^2)}
&\lesssim&\||x|^{\mu+1} (v+w)^{1+q}\|_{L^{4}_T(L^4)}\|e^{(1+\varepsilon)|v+w|^2}-1\|_{L^{\frac43}_T(L^4)}+\||x|^{\mu+1}(v+w)^q\|_{L^1_T(L^2)}.
\end{eqnarray*}
Moreover, since $8q>2+8(1+\mu)$,
$$\||x|^{\mu+1} (v+w)^{1+q}\|_4\lesssim \|v+w\|_{H^1}^{1+q}.$$
We have also
$$\||x|^{\mu+1} (v+w)^q\|\leq \|x(v+w)\|_4\||x|^{\mu} (v+w)^{q-1}\|_4\lesssim \|x(v+w)\|_4\|v+w\|_{H^1}^{q-1}.$$
Thus
\begin{eqnarray*}
\|x|x|^\mu g(v+w)\|_{L^1_T(L^2)}
&\lesssim&T^\frac14\|v+w\|_{L_T^\infty(H^1)}^{q+1}\|e^{(1+\varepsilon)|v+w|^2}-1\|_{L^{\frac43}_T(L^4)}\\
&+&T^\frac34\|x(v+w)\|_{L^4_T(L^4)}\|v+w\|_{L_T^\infty(H^1)}^{q-1}.
\end{eqnarray*}
Since, with a continuity argument, for small positive time $\|\nabla (v+w)\|\leq r+\|\nabla w\|\leq 2r+\|\nabla u_0\|$. We can find $r,\varepsilon>0$ small such that $(1+\varepsilon)\|\nabla (v+w)\|^2<4\pi$. So, by Lemma \ref{exp}, it follows that
\begin{eqnarray*}
\|x|x|^\mu g(v+w)\|_{L^1_T(L^2)}
&\lesssim&T^\frac12\|v+w\|_T^{q}[T^{\frac14}+\|v+w\|_{L_T^4(W^{1,4})}^\alpha T^{\frac34(1-\frac\alpha4)}].
\end{eqnarray*}
It remains to control $\||x|^\mu g(v+w)\|_{L^1_T(\dot H^1)}$. By Lemma \ref{g}, for any $\varepsilon>0,$
\begin{eqnarray*}
\|\nabla [|x|^\mu g(v+w)]\|_{L^1_T(L^2)}
&\lesssim&\||x|^{\mu-1}|v+w|^qe^{(1+\varepsilon)|v+w|^2}\|_{L^1_T(L^2)}\\
&+&\|\nabla(v+w)|x|^\mu(v+w)^{q-1}e^{(1+\varepsilon)|v+w|^2}\|_{L^1_T(L^2)}\\
&=&(I)+(II).
\end{eqnarray*}
Taking $\mu-1$ rather that $\mu$ in previous computations, yields for some real number near four $\alpha\in (0,4)$,
$$(I)\lesssim T^\frac1{4}\|v+w\|_T^{q}[T^{\frac34}+\|v+w\|_{L_T^4(W^{1,4})}^\alpha T^{\frac34(1-\frac\alpha4)}].$$
By H\"older inequality, for any $\varepsilon>0$,
\begin{eqnarray*}
(II)
&\lesssim&\|\nabla(v+w)|x|^{\mu}|v+w|^{q-1}e^{(1+\varepsilon)|v+w|^2}\|_{L^1_T(L^2)}\\
&\lesssim&\|\nabla(v+w)\|_{L^4_T(L^4)}\Big[\||x|^{\mu}|v+w|^{q-1}(e^{(1+\varepsilon)|v+w|^2}-1)\|_{L^\frac43_T(L^4)}+\||x|^{\mu}|v+w|^{q-1}\|_{L^\frac43_T(L^4)}\Big]\\
&\lesssim&\|v+w\|_T\Big[\||x|^{\mu}|v+w|^{q-1}\|_{L_T^\infty(L^{\frac{4(4+\varepsilon)}\varepsilon})}\|e^{(1+\varepsilon)|v+w|^2}-1\|_{L^\frac43_T(L^{4+\varepsilon})}+T^\frac34\|v+w\|_T^{q-1}\Big].
\end{eqnarray*}
Arguing as in the proof of Lemma \ref{exp}, it is sufficient to estimate $\||x|^{\mu}(v+w)^{q-1}\|_{L_T^\infty(L^{\frac{4(4+\varepsilon)}\varepsilon})}$. Since $q>1+2\mu+\frac{2\varepsilon}{4(4+\varepsilon)}$, we have via Lemma \ref{sblv}, $$\||x|^{\mu}(v+w)^{q-1}\|_{L_T^\infty(L^{\frac{4(4+\varepsilon)}\varepsilon})}\lesssim \|v+w\|_{L_T^\infty(H^1)}^{q-1}.$$
Thus,
$$(II)\lesssim [T^{\frac34}+\|v+w\|_{L_T^4(W^{1,4})}^\alpha T^{\frac34(1-\frac\alpha4)}]\|v+w\|_T^{q}.$$
 Now, $\|v+w\|_T\leq r+\|w_0\|_\Sigma=r+\|u_0\|_\Sigma$. Therefore, for $0<T$ small enough,
$$\|\phi(v)\|_{T}
\lesssim [T^{\frac34}+(r+\|u_0\|_{H^1})^\alpha T^{\frac34(1-\frac\alpha4)}](r+\|u_0\|_\Sigma)^{q}.$$
Thus, for $r,T>0$ small enough, $\phi$ maps $B_T(r)$ into itself. It remains to prove that $\phi$ is a contraction. Let $v_1,v_2\in B_T(r)$ solutions to \eqref{eq1}, $u:=v_1-v_2$ and $u_i:=v_i+w$, $i\in\{1,2\}$.
Using Strichartz estimate, for $T\in(0,1)$,
$$\|\phi(v_1)-\phi(v_2)\|_{T}\lesssim \||x|^\mu(g(v_1+w)-g(v_2+w))\|_{L^1_T(\Sigma)}.$$
Using Lemma \ref{g} via H\"older inequality, and the identity $r^{q-1}e^{(1+\varepsilon)r^2}\leq C_\varepsilon(r^{q-1}+r^{q+2}e^{(1+\varepsilon)r^2})$, yields for all $\varepsilon>0$,
\begin{eqnarray*}
\|x|x|^\mu(g(u_1)-g(u_2))\|_{L^1_T(L^2)}
&\lesssim&\sum_{i=1}^2\||x|^{1+\mu} uu_i^{q-1}e^{(1+\varepsilon)|u_i|^2}\|_{L^1_T(L^2)}\\
&\lesssim&\||x|^{1+\mu}u_1^{q+2}u\|_{L^4_T(L^4)}\|e^{(1+\varepsilon)|u_1|^2}-1\|_{L^{\frac43}_T(L^4)}+\||x|^{1+\mu} u_1^{q+2}u\|_{L^1_T(L^2)}\\&+&\||x|^{1+\mu} u_1^{q-1}u\|_{L^1_T(L^2)}.
\end{eqnarray*}
Now, since $4(q+2)>2+4(\mu+1)$, we have 
$$\||x|^{1+\mu} u_1^{q+2}u\|_{L^1_T(L^2)}\lesssim T\|u\|_{L_T^\infty(H^1)}\|u_1\|_{L_T^\infty(H^1)}^{q+2}.$$
$4(q-1)>2+8\mu$ implies that
$$\||x|^{1+\mu} u_1^{q-1}u\|_{L^1_T(L^2)}\leq \|xu\|_{L_T^4(L^4)}\||x|^\mu u_1^{q-1}\|_{L_T^\frac43(L^4)}\lesssim T^\frac34\|xu\|_{L_T^4(L^4)}\|u_1\|_{L_T^\infty(H^1)}^{q-1},$$
$6(q+2)>2+12(1+\mu)$, yields
$$\||x|^{1+\mu} u_1^{q+2}u\|_{L^4_T(L^4)}\leq \|u\|_{L_T^\infty(H^1)}\||x|^{1+\mu} u_1^{q+2}\|_{L_T^4(L^6)}\lesssim T^\frac14\|u\|_{L_T^\infty(H^1)}\|u_1\|_{L_T^\infty(H^1)}^{q+2}.$$
Thus, with Lemma \ref{exp},
\begin{eqnarray*}
\|x|x|^\mu(g(u_1)-g(u_2))\|_{L^1_T(L^2)}
&\lesssim&T^\frac14\|u\|_T\|u_1\|_{L_T^\infty(H^1)}^{q+2}\Big[\|e^{(1+\varepsilon)|u_1|^2}-1\|_{L^{\frac43}_T(L^4)}+1\Big]\\
&\lesssim& T^\frac14(1+\|u_0\|_\Sigma)^{q+2}[1+(1+\|u_0\|_\Sigma)^\alpha T^{\frac34(1-\frac\alpha4)}]\|u\|_T.
\end{eqnarray*}
Compute, fore $|x|^\mu(g(u_1)-g(u_2)):=h$ and $u:=u_1-u_2$,
{\small \begin{eqnarray*}
\|\nabla h\|_{L^1_T(L^2)}
&=&\||x|^\mu(\mathcal Dg(u_1)-\mathcal Dg(u_2))\nabla u_1 +|x|^\mu\mathcal Dg(u_2)\nabla u+\mu|x|^{\mu-1}(g(u_1)-g(u_2))\|_{L^1_T(L^2)}\\
&\leq&\||x|^\mu(\mathcal Dg(u_1)-\mathcal Dg(u_2))\nabla u_1 \|_{L^1_T(L^2)}+\||x|^\mu\mathcal Dg(u_2)\nabla u\|_{L^1_T(L^2)}+\mu\||x|^{\mu-1}(g(u_1)-g(u_2))\|_{L^1_T(L^2)}\\
&:=& (A)+(B)+(C).
\end{eqnarray*}}
By H\"older inequality via Lemma \eqref{exp}, for any $\varepsilon>0$,
\begin{eqnarray*}
(B)
&\lesssim&\|\nabla u|x|^\mu u_2^{q-1}e^{(1+\varepsilon)|u_2|^2}\|_{L^1_T(L^2)}\\
&\lesssim& \|\nabla u\|_{L_T^4(L^4)}\||x|^\mu u_2^{q-1}e^{(1+\varepsilon)|u_2|^2}\|_{L^\frac43_T(L^4)}\\
&\lesssim& \|\nabla u\|_{L_T^4(L^4)}\Big[\||x|^\mu u_2^{q-1}\|_{L_T^\infty(L^{\frac{4(4+\varepsilon)}\varepsilon})}\|e^{(1+\varepsilon)|u_2|^2}-1\|_{L^\frac43_T(L^{4+\varepsilon})}+T^\frac34\|u\|_T^{q-1}\Big].
\end{eqnarray*}
With previous computation, we have
$$(B)\lesssim (1+\|u_0\|_\Sigma)^{q-1}[T^{\frac34}+(1+\|u_0\|_\Sigma)^\alpha T^{\frac34(1-\frac\alpha4)}]\|u\|_T.$$
Let estimate $(C)$. Taking account of Lemma \ref{g} via Sobolev and H\"older inequalities, for any $\varepsilon>0$,
\begin{eqnarray*}
(C)
&\lesssim&\|u|x|^{\mu-1}u_1^{q-1}e^{(1+\varepsilon)|u_1|^2}\|_{L^1_T(L^2)}\\
&\lesssim&\|u|x|^{\mu-1}u_1^{q-1}\|_{L^4_T(L^4)}\|e^{(1+\varepsilon)|u_1|^2}-1\|_{L^{\frac43}_T(L^4)}+\|u|x|^{\mu-1}u_1^{q-1}\|_{L^1_T(L^2)}.
\end{eqnarray*}
By H\"older inequality, via Sobolev injection and the fact that $q>2+2\mu>1+\frac2{4+\varepsilon}+2\mu$,
$$\|u|x|^{\mu-1}u_1^{q-1}\|_{L^4_T(L^4)}\lesssim\|u\|_{L_T^\infty(H^1)}\||x|^{\mu-1}u_1^{q-1}\|_{L^4_T(L^{4+\varepsilon})}\lesssim\|u\|_T\|u_1\|_T^{q-1}T^\frac14.$$
With Lemma \ref{exp}, for some $\alpha\in (0,4)$,
\begin{eqnarray*}
(C)
&\lesssim& T^{\frac{1}4}[T^{\frac34}+\|u_1\|_{L_T^4(W^{1,4})}^\alpha T^{\frac34(1-\frac\alpha4)}]\|u_1\|_T^{q-1}\|u\|_T\\
&\lesssim& T^{\frac{1}4}[T^{\frac34}+(1+\|u_0\|_\Sigma)^\alpha T^{\frac34(1-\frac\alpha4)}](1+\|u_0\|_\Sigma)^{q-1}\|u\|_T.
\end{eqnarray*}
Let control $(A)$. By Lemmas \ref{g}-\ref{exp}, for some large real number $p>1$  such that $q>2+\frac2p+2\mu$ and any $\varepsilon>0$,
\begin{eqnarray*}
(A)
&\lesssim&\sum_{i=1}^2\|\nabla u_1|x|^{\mu}uu_i^{q-2}e^{(1+\varepsilon)|u_i|^2}\|_{L^1_T(L^2)}\\
&\lesssim&\|\nabla u_1\|_{L^4_T(L^4)}\Big[\||x|^{\mu}uu_1^{q-2}\|_{L_T^\infty(L^\frac{4(4+\varepsilon)}\varepsilon)}\|e^{(1+\varepsilon)|u_1|^2}-1\|_{L^\frac43_T(L^{4+\varepsilon})}+T^\frac34\||x|^{\mu}uu_1^{q-2}\|_{L_T^\infty(L^4)}\Big]\\
&\lesssim&\|u_1\|_{T}\|u\|_{L_T^\infty(H^1)}\||x|^{\mu}u_1^{q-2}\|_{L_T^\infty(L^p)}\Big[\|e^{(1+\varepsilon)|u_1|^2}-1\|_{L^\frac43_T(L^{4+\varepsilon})}+T^\frac34\Big]\\
&\lesssim&\|u\|_T\|u_1\|_T^{q-1}\Big[\|e^{(1+\varepsilon)|u_1|^2}-1\|_{L^\frac43_T(L^{4+\varepsilon})}+1\Big].
\end{eqnarray*}
Therefore
$$(A)\lesssim [T^{\frac34}+\|u_1\|_{L_T^4(W^{1,4})}^\alpha T^{\frac34(1-\frac\alpha4)}]\|u_1\|_T^{q-1}\|u\|_T\lesssim [T^{\frac34}+(r+\|u_0\|_\Sigma)^\alpha T^{\frac34(1-\frac\alpha4)}](r+\|u_0\|_\Sigma)^{q-1}\|u\|_T.$$
Finally, for some $\alpha<4$ near four,
$$\|\phi(v_1)-\phi(v_2)\|_T\leq C[T^{\frac34}+(r+\|u_0\|_\Sigma)^\alpha T^{\frac34(1-\frac\alpha4)}](r+\|u_0\|_{H^1})^{q-1}\|v_1-v_2\|_T.$$
So $\phi$ is a contraction of $B_T(r)$ for some $T,r>0$ small enough. It's fix point $v$ satisfies $u=v+w$ is a solution to \eqref{eq1}.
\subsection{Uniqueness in the conformal space}
 Let $u_1,u_2\in C_T(\Sigma)$ two solutions to \eqref{eq1} and $u:=u_1-u_2$. So
$$i\dot u+\Delta u-|x|^2u=|x|^\mu(g(u_1)-g(u_2)),\quad u(0,.)=0.$$
With a continuity argument, there exists $0<T<1$, such that
$$\max_{i\in\{1,2\}}\|\nabla u_i\|_{L^\infty_T(L^2)}^2<4\pi\quad\mbox{and}\quad\max_{i\in\{1,2\}}\|u_i\|_{L^\infty_T(\Sigma)}\leq 1+\|u_0\|_{\Sigma}.$$
Uniqueness follows from previous computation with the result
\begin{lem}\mbox{}\\
\begin{enumerate}
\item
$\|\nabla u_1\|_{L^4_T(L^4)}\lesssim(1+\|u_0\|_{\Sigma})^{q}.$
\item
$\|xu\|_{L^{\infty}_T(L^2)\cap L^4_T(L^4)}\lesssim\frac{(1+\|u_0\|_\Sigma)^{q+1+\alpha q}+T^\frac34}{1-(1+\|u_0\|_\Sigma)^{q-1}T^\frac34}T^\frac14\|u\|_{L_T^\infty(H^1)}.$
\item
$\|\nabla u\|_{L^4_T(L^4)}\lesssim
\frac{(1+\|u_0\|_\Sigma)^{q-2}\Big[(T^\frac14+\|\nabla u_1\|_{L_T^4(L^4)}(T^\frac34+(1+\|u_0\|_\Sigma)^{q\alpha}T^{\frac34(1-\frac\alpha4)}))\Big]}{1-(1+\|u_0\|_\Sigma)^{q-2}(T^\frac34+(1+\|u_0\|_\Sigma)^{q\alpha}T^{\frac34(1-\frac\alpha4)})}\|u\|_{L_T^\infty(H^1)}.$
\end{enumerate}
\end{lem}
\begin{proof}
With previous computation, for small $T>0$,
\begin{eqnarray*}
\|u_1\|_{L_T^4(W^{1,4})}
&\lesssim&\|u_0\|_\Sigma+[T^\frac14\|u_1\|_{L_T^\infty(H^1)}^q+\|u_1\|_{L_T^4(W^{1,4})}\|u_1\|_{L_T^\infty(H^1)}^{q-1}][T^\frac34+\|u_1\|_{L_T^4(W^{1,4})}^\alpha T^{\frac34(1-\frac\alpha4)}]\\
&\lesssim&\|u_0\|_\Sigma+T^\frac14\|u_1\|_{L_T^\infty(H^1)}^q[1+\|u_1\|_{L_T^4(W^{1,4})}][1+\|u_1\|_{L_T^4(W^{1,4})}^\alpha]\\
&\lesssim&\|u_0\|_\Sigma+T^\frac14\|u_1\|_{L_T^\infty(H^1)}^q[\|u_1\|_{L_T^4(W^{1,4})}^{\alpha+1}+1]\\
&\lesssim&1+\|u_0\|_\Sigma^q+T^\frac14(1+\|u_0\|_\Sigma)^q\|u_1\|_{L_T^4(W^{1,4})}^{\alpha+1}.
\end{eqnarray*}
With Lemma \ref{abs}, yields $\|u_1\|_{L_T^4(W^{1,4})}\lesssim 1+\|u_0\|_\Sigma^q$.
Using Moser-Trudinger inequality via Lemma \ref{g'}, we have for any $\varepsilon>0$,
\begin{eqnarray*}
\|xu\|_{L^{\infty}_T(L^2)\cap L^4_T(L^4)}
&\lesssim&\|x|x|^\mu(g(u_1)-g(u_2))\|_{L^1_T(L^2)}+T\|\nabla u\|_{L^{\infty}_T(L^2)}\\
&\lesssim&\sum_i (\|u|x|^{1+\mu} u_i^{q+1}(e^{(1+\varepsilon)|u_i|^2}-1)\|_{L^1_T(L^2)}+\|xu|x|^\mu u_i^{q-1}\|_{L_T^1(L^2)})+T\|u\|_{L^{\infty}_T(H^1)}\\
&\lesssim&\|u|x|^{1+\mu} u_1^{q+1}\|_{L^4_T(L^4)}\|e^{(1+\varepsilon)|u_1|^2}-1\|_{L^\frac43_T(L^4)}+\|xu|x|^\mu u_1^{q-1}\|_{L_T^1(L^2)}+T\|u\|_{L^{\infty}_T(H^1)}\\
&\lesssim&\|u\|_{L_T^\infty(H^1)}\|u_1\|_{L_T^\infty(H^1)}^{q+1}T^\frac14\|e^{(1+\varepsilon)|u_1|^2}-1\|_{L^\frac43_T(L^4)}+\|xu\|_{L_T^4(L^4)}\|u_1\|_{L_T^\infty(H^1)}^{q-1}T^\frac34+T\|u\|_{L^{\infty}_T(H^1)}\\
&\lesssim&\frac{\|u_1\|_{L_T^\infty(H^1)}^{q+1}\|e^{(1+\varepsilon)|u_1|^2}-1\|_{L^\frac43_T(L^4)}+T^\frac34}{1-(1+\|u_0\|_\Sigma)^{q-1}T^\frac34}T^\frac14\|u\|_{L_T^\infty(H^1)}\\
&\lesssim&\frac{(1+\|u_0\|_\Sigma)^{q+1+\alpha q}+T^\frac34}{1-(1+\|u_0\|_\Sigma)^{q-1}T^\frac34}T^\frac14\|u\|_{L_T^\infty(H^1)}.
\end{eqnarray*} 
With previous computation
{\small\begin{eqnarray*}
\|\nabla u\|_{L^4_T(L^4)}
&\lesssim&(1+\|u_0\|_\Sigma)^{q-2}\Big[\|\nabla u\|_{L^4_T(L^4)}(T^\frac34+\|e^{(1+\varepsilon)|u_1|^2}-1\|_{L_T^\frac43(L^4)})\\
&+&\|u\|_{L_T^\infty(H^1)}(T^\frac14+\|\nabla u_1\|_{L_T^4(L^4)}(T^\frac34+\|e^{(1+\varepsilon)|u_1|^2}-1\|_{L_T^\frac43(L^4)}))\Big]\\
&\lesssim&(1+\|u_0\|_\Sigma)^{q-2}\Big[\|\nabla u\|_{L^4_T(L^4)}(T^\frac34+(1+\|u_0\|_\Sigma)^{q\alpha}T^{\frac34(1-\frac\alpha4)})\\
&+&\|u\|_{L_T^\infty(H^1)}(T^\frac14+\|\nabla u_1\|_{L_T^4(L^4)}(T^\frac34+(1+\|u_0\|_\Sigma)^{q\alpha})T^{\frac34(1-\frac\alpha4)})\Big]\\
&\lesssim&\frac{(1+\|u_0\|_\Sigma)^{q-2}\Big[(T^\frac14+\|\nabla u_1\|_{L_T^4(L^4)}(T^\frac34+(1+\|u_0\|_\Sigma)^{q\alpha}T^{\frac34(1-\frac\alpha4)}))]}{1-(1+\|u_0\|_\Sigma)^{q-2}(T^\frac34+(1+\|u_0\|_\Sigma)^{q\alpha}T^{\frac34(1-\frac\alpha4)})}\|u\|_{L_T^\infty(H^1)}.
\end{eqnarray*}}
\end{proof}
\subsection{Global well-posedness in the defocusing case}
This subsection is devoted to prove that the maximal solution of \eqref{eq1} is global if $\epsilon=-1$ and $E(u_0)\leq4\pi$. Recall an important fact that is the time of local existence depends only on
the quantity $\|u_0\|_\Sigma$. Let $u$ to be the unique maximal
solution of \eqref{eq1} in the space ${\mathcal{E}}_T$ for any $0<T<{T^*}$ with initial data $u_0$,
where $0<T^*\leq +\infty$ is the lifespan of $u$. We shall prove that $u$ is global.
By contradiction, suppose that $T^*< +\infty$. Consider for $0<s<T^*$, the problem
$$({\mathcal{P}}_{s})\left\{\begin{array}{cccc}
i\partial_{t}v+\Delta{v}-|x|^2v&=&|x|^\mu g(v)\\
v(s,.)&=&u(s,.).
 \end{array}\right.$$
First, let treat the simplest case $E(u_0)<4\pi$. In this case we have
$$\sup_{[0,T^*]}\|\nabla u(t)\|^2\leq E(u_0)<4\pi.$$
 Using the same arguments used in the local existence, we can find a real $\tau>0$ and a solution
$v$ to $({\mathcal{P}}_{s})$ on $[s,s+\tau]$. According to the section of local existence,
 and using the conservation of energy, $\tau$ does not depend on $s$. Thus, if we let $s$ be
close to $T^*$ such that $s+\tau>T^*$, we can extend $v$ for times higher than $T^*$. This fact
contradicts the maximality of $T^*$. We obtain the result claimed.\\
Second, let treat the limit case
$$E=4\pi\quad\mbox{ and }\quad \sup_{[0,T^*]}\|\nabla u(t)\|^2=\limsup_{T^*}\|\nabla u(t)\|^2=4\pi.$$
Then, since near zero $G(r^2)\simeq r^{q+1}$ and $q>2+2\mu$, then $r^{3+2\mu}\lesssim |G(r^2)|$, thus
$$\liminf_{T^*}\int|x|^\mu G(|u(t)|^2)\,dx=\liminf_{T^*}\int|x|^\mu |u(t)|^{3+2\mu}\,dx=0.$$
Global well-posedness is a consequence of The following result.
\begin{lem}\label{cnt}
Let $T>0$ and $u\in C([0,T],\Sigma)$ a solution to the Schr\"odinger equation \eqref{eq1} with $\epsilon=-1$
such that $E(u_0)+M(u_0)<\infty$. Then, a positive constant $C_0$ depending on $u_0$ exists such that for any $R,R'>0$ and any $0<t<T$ we have
\begin{equation}\label{R}
\int_{B_{R+R'}}|u(t)|^2dx\geq \int_{B_R}|u_0|^2dx -C_0\frac{t}{R'}.
\end{equation}
\end{lem}
\begin{proof}[Proof of Lemma \ref{cnt}]
Let $R,R'>0$, $d_R(x):=d(x,B_R)$ and a cut-off function $\phi:=h(1-\frac{d_R}{R'})$, where
$h\in C^{\infty}(\R)$, $h(t)=1$ for $t\geq 1$ and $ h(t)=0$ for $t\leq 0.$ So $\phi(x)=1$ for $x\in B_R$ and $\phi(x)=0$ for $x\notin B_{R+R'}$. Moreover,
$$\nabla\phi (x)=-\frac{x}{R'|x|}h'(1-\frac{d_R(x)}{R'})=-\frac{x}{R'|x|}h'(1-\frac{d_R(x)}{R'}){\bf 1}_{\{R<|x|<R+R'\}}.$$
$$\|\nabla \phi\|_{L^{\infty}}\leq\frac{\|h'\|_{L^{\infty}([0,1])}}{R'}\lesssim\frac{1}{R'}.$$
Multiplying \eqref{eq1} by $\phi^2\bar u$, we obtain
$$\phi^2\bar u(i\dot u+\Delta u-|x|^2 u)=\phi^2|u|^2G'(|u|^2).$$
Integrating on space then taking imaginary part, yields
\begin{eqnarray*}
\partial_t\|\phi u\|^2
&=&-2\Im \int\phi^2\bar u\Delta u dx\\
&=&2\Im \int\nabla(\phi^2\bar u)\nabla u dx\\
&=&4\Im \int(\phi\nabla\phi\bar u\nabla u) dx\geq-\frac{C_0}{R'}.
\end{eqnarray*}
An integrating on time achieves the proof.
\end{proof}
We return to the proof of global well-posedness. With H\"older inequality via \eqref{R}, for any $p\geq1$,
\begin{eqnarray*}
\int_{B_R}|u_0|^2dx -C_0\frac{t}{R'}
&\leq&\||x|^\frac{\mu}p|u(t)|^2\|_{L^p({B_{R+R'}})}\||x|^{-\frac{\mu}p}\|_{L^{p'}({B_{R+R'}})}\\
&\leq&\sqrt\pi\Big(\frac{(R+R')^{2-\frac{p'\mu}p}}{2-\frac{p'\mu}p}\Big)^\frac1{p'}\Big(\int_{B_{R+R'}}|x|^\mu|u(t)|^{2p}\,dx\Big)^\frac1p.
\end{eqnarray*}
Choose $p=\frac32+\mu$ so that $\frac{p'\mu}p<2$. Taking the lower limit when $t$ tends to $T^*$, then $R'\rightarrow\infty$ yields to the contradiction $u_0=0$.
This ends the proof.
\subsection{Local solution in the critical case}
In this subsection, we prove Theorem \ref{per}, about existence of a local solution to \eqref{eq1} in the critical case and without any condition on the data size.\\
 Let $u_0\in \Sigma$. Using Littlewood-Paley theory, we decompose the data  $u_0=(u_0)_{<n}+(u_0)_{>n}$ such that $(u_0)_{>n}\rightarrow0$ in $H^1$ and $(u_0)_{<n}\in H^2\cap \Sigma.$
First, consider the Cauchy problem
$$i\dot v+\Delta v-|x|^2v+|x|^\mu g(v)=0, \quad v_{|t=0}=(u_0)_{<n}.$$
Using the embedding $H^2\hookrightarrow L^\infty$, it is easy to find $0<T_n$ and $v\in C_{T_n}(H^2\cap\Sigma)$ a solution to the previous problem. Arguing as in the previous section, by a standard fixed point argument, we find a solution to the perturbed Cauchy problem
$$i\dot w+\Delta w-|x|^2 w=|x|^\mu(g(v+w)-g(v)), \quad w_{|t=0}=(u_0)_{>n}.$$
in the space $C_T(\Sigma)\cap L_T^4(W^{1,4})$. Thus $u:=v+w$ is a solution to \eqref{eq1}.
\section{The stationary problem}
The goal of this section is to prove Theorem \ref{tgs} about existence of a ground state solution to the stationary problem associated to \eqref{eq1}. Precisely, we look for a minimizing of \eqref{min} which is a solution to
\begin{equation}\label{gr}
-\Delta\phi+|x|^2\phi+\phi=|x|^\mu\phi G'(|\phi|^2):=|x|^\mu g(\phi),\quad0\neq\phi\in {\Sigma}.
\end{equation}
We assume in all this section that \eqref{f} is satisfied and we prove that \eqref{gr} has a ground state in the meaning that it has a nontrivial positive radial solution which minimizes the action $S$ when $K_{\alpha,\beta}$ vanishes. 
\begin{rem}\mbox{}\\
\begin{enumerate}
\item
In all this section, we are concerned with the focusing case, so we take $\epsilon=1$.
\item
If $\phi$ is a solution to \eqref{gr}, then $e^{it}\phi$ is a solution to the Schr\"odinger problem \eqref{eq1} with data $\phi$. This particular global solution said soliton or standing wave does not satisfy neither decay nor scattering.
\end{enumerate}
\end{rem}
Here and hereafter, we denote the quadratic part and the nonlinear parts of $K_{\alpha,\beta}$,
\begin{gather*}
K^Q_{\alpha,\beta}(v):=2\int\Big[\alpha|\nabla v|^2+(\alpha+\beta)|v|^2+(\alpha+2\beta)|xv|^2\Big]\,dx,\\
K^N_{\alpha,\beta}(v):=-2\int|x|^\mu\Big[\alpha |v|g(|v|)+\beta(1+\frac\mu2) G(|v|^2)\Big]\,dx.
\end{gather*}
We denote also the operator $\mathcal L_{\alpha,\beta}S(v):=\partial_{\lambda}(S(v_{\alpha,\beta}^{\lambda}))_{|\lambda=0}$ and
\begin{eqnarray*}
H_{\alpha,\beta}(v)
&:=&(1-\frac{\mathcal L_{\alpha,\beta}}{2(\alpha+2\beta)})S(v)\\
&=&\frac1{\alpha+2\beta}\int\Big[\beta(2|\nabla v|^2+|v|^2)+\alpha  |v|g(|v|)-(\alpha-\beta(\frac\mu2-1))G(|v|^2)\;dx\Big]\\
&=&\frac1{\alpha+2\beta}\Big[\beta(2\|\nabla v\|^2+\|v\|^2)+\int|x|^\mu\Big[ \alpha D-(\alpha-\beta(\frac\mu2-1))\Big]G(|v|^2)\;dx.
\end{eqnarray*}
The proof of Theorem \ref{tgs} is approached via a series of lemmas. Let us start by the so-called generalized Pohozaev identity, which is a useful classical result about solution to \eqref{gr}.
\begin{prop}\label{poh}
If $\phi$ is a solution to \eqref{gr}, then for any $\alpha,\beta\in\R$, we have
$$\alpha\|\nabla\phi\|^2+(\alpha+\beta)\|\phi\|^2+(\alpha+2\beta)\|x\phi\|^2-\int|x|^{\mu}\Big[\alpha|\phi|g(|\phi|)+\beta(1+\frac\mu2) G(|\phi|^2)\Big]dx=0.$$
\end{prop}
\begin{proof}
Take the action defined on $\Sigma$ by
$$S(v):=\|v\|_\Sigma^2-\int|x|^\mu G(|v|^2)\,dx.$$
Then, $S'(\phi)=2\langle-\Delta \phi+\phi+|x|^2\phi-|x|^\mu g(\phi),.\rangle=0$ because $\phi$ is a solution to \eqref{gr}. For $\alpha,\beta,\lambda\in\R$ and $v\in\Sigma$, we denote the function $v_{\alpha,\beta}^{\lambda}:=e^{\alpha\lambda}v(e^{-\beta\lambda}.)$ and $K_{\alpha,\beta}(v):=\partial_{\lambda}(S(v_{\alpha,\beta}^{\lambda}))_{|\lambda=0}$. A simple computation yields
\begin{equation}
\frac12K_{\alpha,\beta}(\phi)=\alpha\|\nabla\phi\|^2+(\alpha+\beta)\|\phi\|^2+(\alpha+2\beta)\|x\phi\|^2-\int|x|^\mu\Big[\alpha|\phi|g(|\phi|)+\beta(1+\frac\mu2) G(|\phi|^2)\Big]dx.\end{equation}
Since $\partial_{\lambda}(S(\phi_{\alpha,\beta}^{\lambda}))_{|\lambda=0}=\langle S'(\phi),\partial_{\lambda}(\phi_{\alpha,\beta}^{\lambda})_{|\lambda=0}\rangle=0$, we have $K_{\alpha,\beta}(\phi)=0$.
\end{proof}
We need the following result.
\begin{lem}
Let $(0,0)\neq(\alpha,\beta)\in\R_+^2$ and $0\neq\phi\in\Sigma$. Then, denoting for simplicity $\mathcal L:=\mathcal L_{\alpha,\beta}$ and $\phi^\lambda:=\phi_{\alpha,\beta}^\lambda$, we have
\begin{enumerate}
\item
$\min(\mathcal LH_{\alpha,\beta}(\phi),H_{\alpha,\beta}(\phi))> 0.$
\item
$\lambda\mapsto H_{\alpha,\beta}(\phi^{\lambda})$ is increasing.
\end{enumerate}
\end{lem}
\begin{proof}
If $\alpha=0$, clearly $\mathcal LH_{\alpha,\beta}(\phi)>0$. Else, with \eqref{f},
\begin{eqnarray*}
H_{\alpha,\beta}(\phi)
&=&\frac1{\alpha+2\beta}\Big[\beta(2\|\nabla \phi\|^2+\|\phi\|^2)+\alpha\int\Big( |\phi|g(|\phi|)-(1-\frac\beta\alpha(\frac\mu2-1))G(|\phi|^2)\Big)\;dx\Big]\\
&=&\frac1{\alpha+2\beta}\Big[\beta(2\|\nabla \phi\|^2+\|\phi\|^2)+\alpha\int\Big(D-1+\frac\beta\alpha(\frac\mu2-1)\Big)G(|\phi|^2)\;dx\Big]>0.
\end{eqnarray*}
Moreover, with a direct computation
\begin{eqnarray*}
\mathcal LH_{\alpha,\beta}(\phi)
&=&\mathcal L(1-\frac1{2(\alpha+2\beta)}\mathcal L)S(\phi)\\
&=&-\frac1{2(\alpha+2\beta)}(\mathcal L-2\alpha)(\mathcal L-2(\alpha+2\beta))S(\phi)+2\alpha(1-\frac1{2(\alpha+2\beta)}\mathcal L)S(\phi)\\
&=&-\frac1{2(\alpha+2\beta)}(\mathcal L-2\alpha)(\mathcal L-2(\alpha+2\beta))S(\phi)+2\alpha H_{\alpha,\beta}(\phi).
\end{eqnarray*}
Now, since $(\mathcal L-2\alpha)\|\nabla\phi\|^2=(\mathcal L-2(\alpha+2\beta))\|x\phi\|^2=0$, we have $(\mathcal L-2\alpha)(\mathcal L-2(\alpha+2\beta))[\|\nabla\phi\|^2+\|x\phi\|^2]=0$. Moreover, $\mathcal L[|x|^\mu G(|\phi|^2)]=2|x|^\mu[(\alpha D+\beta(1+\frac\mu2))G](|\phi|^2)$, so
\begin{eqnarray*}
\mathcal LH_{\alpha,\beta}(\phi)
&\geq&\frac1{2(\alpha+2\beta)}\int(\mathcal L-2\alpha)(\mathcal L-2(\alpha+2\beta))|x|^\mu G(|\phi|^2)dx\\
&=&\frac2{\alpha+2\beta}\int(|x|^\mu[\alpha(D-1)+\beta(1+\frac\mu2)][\alpha(D-1)+\beta(\frac\mu2-1)]G(|\phi|^2)\,dx)>0.
\end{eqnarray*}
The last inequality comes from \eqref{f}. The two first points of the Lemma follow. The third point is a consequence of the equality $\partial_{\lambda}H_{\alpha,\beta}(\phi^{\lambda})=\mathcal LH_{\alpha,\beta}(\phi^{\lambda})$.
\end{proof}
The next intermediate result is the following
\begin{lem}\label{lm1}
Let $(0,0)\neq(\alpha,\beta)\in\R_+^2$. Take a bounded sequence  $0\neq\phi_n\in {\Sigma}$ such that $\displaystyle\lim_nK^Q_{\alpha,\beta}(\phi_n)=0$. Then, there exists $n_0\in\N$ such that
$K_{\alpha,\beta}(\phi_n)>0$ for all $n\geq n_0.$
\end{lem}
\begin{proof}[Proof of Lemma \ref{lm1}]
 Using Lemma \ref{exp} via \eqref{f}, there exists $q>3+2\mu$ and $a>0$ such that $\displaystyle\sup_{r\geq0}\frac{|rg(r)+G(r^2)|}{r^{q-1}(e^{ar^2}-1)}\lesssim 1$. By H\"older inequality, for any $p\geq 1$,
\begin{eqnarray*}
K^N_{\alpha,\beta}(\phi_n)
&\lesssim&\int|x|^\mu|\phi_n|^{q-1}(e^{a|\phi_n|^2}-1)dx\\
&\lesssim&\||x|^\mu\phi_n^{q-1}\|_p\|e^{a|\phi_n|^2}-1\|_{p'}\\
&\lesssim&\||x|^\mu\phi_n^{q-1}\|_p\|e^{ap'|\phi_n|^2}-1\|_{1}^{\frac1{p'}}.
\end{eqnarray*}
Taking $p$ satisfying $a{p'}\|\phi_n\|_{H^1}^2<2\pi$, by Moser-Trudinger inequality $K^N_{\alpha,\beta}(\phi_n)\lesssim\||x|^\mu\phi_n^{q-1}\|_p\|\phi_n\|^\frac2{p'}.$ Using the interpolation inequality in Lemma \ref{sblv} via the fact that $q-1>2\mu+\frac2p$, we have
$$K^N_{\alpha,\beta}(\phi_n)
\lesssim\||x|^{p\mu}\phi_n^{p(q-1)}\|_1^{\frac1p}\|\phi_n\|^\frac2{p'}\lesssim \|\phi_n\|^{\mu+2}\|\nabla \phi_n\|^{q-1-\mu-\frac2p}.$$
If $\alpha\neq0$, the proof is achieved by the fact that $\|\nabla\phi_n\|^2\lesssim K^Q_{\alpha,\beta}(\phi_n)$ and taking $p>2$ such that $q-1-\mu-\frac2p>2$. If $\alpha=0$, the proof is closed since $\|\phi_n\|^{2+\mu}=o(\|\phi_n\|^2)$ and $\|\phi_n\|^2\lesssim K^Q_{0,\beta}(\phi_n)$.
\end{proof}
The last auxiliary result of this section reads
\begin{lem}\label{lm3}
Let $(0,0)\neq(\alpha,\beta)\in\R_+^2$. 
 Then
\begin{equation}
m_{\alpha,\beta}=\inf_{0\neq\phi\in{\Sigma}}\{H_{\alpha,\beta}(\phi),\;\mbox{s.t}\; K_{\alpha,\beta}(\phi)\leq 0\}.
\end{equation}
\end{lem}
\begin{proof}
Let $m_1$ be the right hand side, it is sufficient to prove that $m\leq m_1$. Take $\phi\in{\Sigma}$ such that $K_{\alpha,\beta}(\phi)<0$ then by Lemma \ref{lm1},  the fact that $\displaystyle\lim_{\lambda\rightarrow-\infty}K_{\alpha,\beta}^Q(\phi^{\lambda}_{\alpha,\beta})=0$ and $\lambda\mapsto H_{\alpha,\beta}(\phi_{\alpha,\beta}^{\lambda})$ is increasing, there exists $\lambda<0$ such that
\begin{equation}\label{0}
K_{\alpha,\beta}(\phi^{\lambda}_{\alpha,\beta})=0,\; H_{\alpha,\beta}(\phi_{\alpha,\beta}^{\lambda})\leq H_{\alpha,\beta}(\phi).
\end{equation}
The proof is closed.
\end{proof}
Now, we are in a position to prove the main result.\\
{\bf Proof of Theorem \ref{tgs}}\\
\underline{First case $\alpha\neq0$ and $g$ subcritical}.\\
Let $(\phi_n)$ a minimizing sequence, namely
\begin{equation}\label{m}
0\neq\phi_n\in{\Sigma},\;K_{\alpha,\beta}(\phi_n)= 0\;\mbox{and}\;\lim_nH_{\alpha,\beta}(\phi_n)=\lim_nS(\phi_n)=m_{\alpha,\beta}.
\end{equation}
Then, $\alpha\Big[\|\phi_n\|_\Sigma^2-\int|x|^\mu|\phi_n|g(|\phi_n|)\;dx\Big]=\beta\Big[(1+\frac\mu2)\int|x|^\mu G(|\phi_n|^2)\,dx-\|\phi_n\|^2-2\|x\phi_n\|^2\Big]$
and $\Big(\|\phi_n\|_\Sigma^2-\int|x|^\mu G(|\phi_n|^2)\Big)\rightarrow m_{\alpha,\beta}$. Denoting $\lambda:=\frac{\beta}{\alpha}$, yields $\|\phi_n\|_\Sigma^2-\int|x|^\mu|\phi_n|g(|\phi_n|)\;dx=\lambda\Big[ \|\nabla\phi_n\|^2-\|x\phi_n\|^2-\|\phi_n\|_\Sigma^2+(1+\frac\mu2)\int |x|^\mu G(|\phi_n|^2)dx\Big]$. Thus
$$\lambda\Big[\|\phi_n\|_\Sigma^2-(1+\frac\mu2)\int|x|^\mu G(|\phi_n|^2)\;dx\Big]=\lambda[\|\nabla\phi_n\|^2-\|x\phi_n\|^2]-\|\phi_n\|_\Sigma^2+\int|x|^\mu|\phi_n|g(|\phi_n|)\,dx.$$
So the following sequences are bounded
\begin{gather*}
\lambda\Big[\|\nabla\phi_n\|^2-\|x\phi_n\|^2+\frac\mu2\int |x|^\mu G(|\phi_n|^2)\,dx\Big]-\|\phi_n\|_\Sigma^2+\int|x|^\mu|\phi_n|g(|\phi_n|)\,dx,\\
\lambda[\|\nabla\phi_n\|^2-\|x\phi_n\|^2]+\int|x|^\mu\Big(|\phi_n|g(|\phi_n|)-(1-\frac{\lambda\mu}2)G(|\phi_n|^2)\Big)\,dx,\\
\lambda\|\nabla\phi_n\|^2+\int|x|^\mu\Big(D-(1-\frac{\lambda\mu}2)G(|\phi_n|^2)\Big)\,dx-\lambda\|x\phi_n\|^2,\\
\lambda\|\nabla\phi_n\|^2+\int|x|^\mu\Big(D-(1+\lambda-\frac{\lambda\mu}2)G(|\phi_n|^2)\Big)\,dx+\lambda\|\phi_n\|_{H^1}^2,\\
\end{gather*}
Suppose that $\beta\neq0$. Thus, using the assumption $(D-1)G\geq 0$, we have $\|\phi_n\|_{H^1}\lesssim 1$. This implies that $(\phi_n)$ is bounded in $\Sigma$, in fact if $\|\phi_n\|_{H^1}\lesssim 1$ and $\|x\phi_n\|\rightarrow\infty$, we have
$$\int|x|^\mu G(|\phi_n|^2)\,dx\geq -m_{\alpha,\beta}-1+\|\phi_n\|_{H^1}^2+\|x\phi_n\|^2\geq C(\|\phi_n\|_{H^1}^2+\|x\phi_n\|^2).$$
By Moser-Trudinger inequality,  and the interpolation inequality in Lemma \ref{sblv}, we obtain, the absurdity
\begin{eqnarray}
\infty \leftarrow\int|x|^\mu G(|\phi_n|^2)dx
&\lesssim&\int|x|^\mu|\phi_n|^{q-1}(e^{\varepsilon|\phi_n|^2}-1)\,dx, \quad\forall\varepsilon>0 \nonumber\\
&\lesssim& \|\phi_n\|_{H^1}^{q-1}.\label{chng}
\end{eqnarray}
Assume now that $\beta=0$ and $(D-1-\varepsilon_g)>0$ for some $\varepsilon_g>0$. Then
$$\|\phi_n\|_\Sigma^2=\int|x|^\mu|\phi_n|g(|\phi_n|)\,dx\quad\mbox{and}\quad\Big(\|\phi_n\|_\Sigma^2-\int|x|^\mu G(|\phi_n|^2)\,dx\Big)\rightarrow m_{\alpha,\beta}.$$
Thus, for any real number $a\neq 0$,
{\small$$\Big((1-a)\|\phi_n\|_\Sigma^2+a\int|x|^\mu[|\phi_n|g(|\phi_n|)-\frac1aG(|\phi_n|^2)]dx\Big)=\Big((1-a)\|\phi_n\|_\Sigma^2+a\int|x|^\mu[D-\frac1a]G(|\phi_n|^2)dx\Big)\rightarrow m_{\alpha,\beta}.$$}
Taking $a:=\frac1{1+\varepsilon_g}$, we conclude that $(\phi_n)$ is bounded in $\Sigma$.\\
 This implies, via the compact injection $H^1_{rd}\hookrightarrow L^p$, for any $2<p<\infty$, that
$$\phi_n\rightharpoonup\phi\quad\mbox{in}\quad \Sigma\quad\mbox{and}\quad\phi_n\rightarrow\phi\quad\mbox{in}\quad L^p,\;\forall p\in (2,\infty).$$
Assume that $\phi=0$. With \eqref{f}, There exists $p>2+2\mu$ and $a>0$ small enough, such that
$$\max\{|G(r^2)|,r|g(r)|\}\lesssim r^p(e^{ar^2}-1).$$
Since $(\phi_n)$ is bounded in $H^1_{rd}$ and using Moser-Trudinger inequality and Lemma \ref{sblv}, yields
\begin{eqnarray*}
\int|x|^\mu\Big( G(|\phi_n|^2)dx+|\phi_n|g(|\phi_n|)\Big)dx
&\lesssim &\||x|^\mu\phi_n^p(e^{a|\phi_n|^2}-1)\|_{1}\\
&\lesssim &\|\phi_n\|_{4}\||x|^\mu\phi_n^{p-1}\|_{4}\|e^{a|\phi_n|^2}-1\| \nonumber\\
&\lesssim&\|\phi_n\|_4\|\phi_n\|_{H^1}^{p-1}\rightarrow0.
\end{eqnarray*}
By Lemma \ref{lm1}, $K_{\alpha,\beta}(\phi_n)>0$ for large $n$ which is absurd. So
\begin{equation}\label{a1}\phi\neq 0.\end{equation}
With lower semi continuity of the $\Sigma$ norm, we have $K_{\alpha,\beta}(\phi)\leq0$ and $H_{\alpha,\beta}(\phi)\leq m$. Using \eqref{0}, we can assume that
$K_{\alpha,\beta}(\phi)=0$ and $S(\phi)=H_{\alpha,\beta}(\phi)\leq m.$ So that $\phi$ is a minimizer satisfying $0\neq\phi\in\Sigma$, $K_{\alpha,\beta}(\phi)=0$ and $S(\phi)=H_{\alpha,\beta}(\phi)= m$. Since
$$0<\int|x|^\mu G(|\phi|^2)\;dx\lesssim H_{\alpha,\beta}(\phi)
=\frac1{\alpha+\beta}\Big[\beta(2\|\nabla \phi\|^2+\|\phi\|^2)+\alpha \int|x|^\mu\Big(D-1-\frac{\beta}{\alpha}(\frac\mu2-1)\Big)G(|\phi|^2)\;dx\Big],$$
we have
$$m_{\alpha,\beta}>0.$$
Now, there is a Lagrange multiplier $\eta\in\R$ such that $S'(\phi)=\eta K'(\phi)$. Denoting $\mathcal L(\phi)=(\partial_{\lambda}\phi_{\alpha,\beta}^{\lambda})_{|\lambda=0}$ and $\mathcal L S(\phi)=(\partial_{\lambda}S(\phi_{\alpha,\beta}^{\lambda}))_{|\lambda=0}$, yields
\begin{eqnarray*}
0=K_{\alpha,\beta}(\phi)&=&\mathcal LS(\phi)=\langle S'(\phi),\mathcal L(\phi)\rangle\\
&=&\eta\langle K'(\phi),\mathcal L(\phi)\rangle\\
&=&\eta\mathcal LK(\phi)=\eta\mathcal L^2S(\phi).\end{eqnarray*}
With a previous computation and taking account of \eqref{f},
\begin{eqnarray*}
-\mathcal L^2S(\phi)-4\alpha(\alpha+2\beta)S(\phi)
&=&-(\mathcal L-2(\alpha+2\beta))(\mathcal L-2\alpha)S(\phi)\\
&=&4\int|x|^\mu[\alpha(D-1)+\beta(1+\frac\mu2)][\alpha(D-1)+\beta(\frac\mu2-1)]G(|\phi|^2)\,dx\\
&>&0.
\end{eqnarray*}
Thus $\eta=0$ and $S'(\phi)=0$. So, $\phi$ is a ground state.
\\\underline{Second case $\alpha\neq0$ and $g$ critical}.\\
The proof is similar to the first case, the only two points to change are \eqref{chng} and \eqref{a1}. Let for $\lambda\in(0,\frac{1}{\alpha_g\sup_n\|\phi_n\|_{H^1}})$, $\phi_{n,\lambda}:=\lambda\phi_n$. It is clear that $\phi_{n,\lambda}$ satisfies \eqref{chng}. Using Moser-Trudinger inequality and Lemma \ref{g}, yields for some $p>2+2\mu$,
\begin{eqnarray}
\int|x|^\mu\Big( G(|\phi_{n,\lambda}|^2)dx+|\phi_{n,\lambda}|g(|\phi_{n,\lambda}|)\Big)dx
&\lesssim &\||x|^\mu\phi_{n,\lambda}^p(e^{\alpha_g|\phi_{n,\lambda}|^2}-1)\|_{1}\nonumber\\
&\lesssim &\|\phi_{n,\lambda}\|_{4}\||x|^\mu\phi_{n,\lambda}^{p-1}\|_{4} \nonumber\\
&\lesssim&\lambda^{3+p}\|\phi_n\|_4\|\phi_n\|_{H^1}^{p-1}\rightarrow0.
\end{eqnarray}
Thus, $K^N(\phi_{n,\lambda})\rightarrow0$ as $n\rightarrow\infty$. So by Lemma \ref{lm1}, $\displaystyle\inf_nK^Q(\phi_{n,\lambda})>0$. Moreover, since $\displaystyle\sup_nK^N_{\alpha,\beta}(\phi_{n,\lambda})=o(\lambda^2)$, when $\lambda$ tends to zero, we have for some small $\lambda_0>0$,
\begin{eqnarray*}
K_{\alpha,\beta}(\phi_{n,\lambda_0})
&=&K^N_{\alpha,\beta}(\phi_{n,\lambda_0})+K^Q_{\alpha,\beta}(\phi_{n,\lambda_0})\\
&=&K^N_{\alpha,\beta}(\phi_{n,\lambda_0})-\lambda_0^2K^Q_{\alpha,\beta}(\phi_{n})\\
&\leq&0.
\end{eqnarray*}
If $K_{\alpha,\beta}(\phi_{n,\lambda_0})<0$ then by Lemma \ref{lm1}, the fact that $\displaystyle\lim_{\lambda\rightarrow0}K_{\alpha,\beta}^Q(\phi_{n,\lambda})=0$ and $\lambda\mapsto H_{\alpha,\beta}(\phi^{\lambda})$ is increasing, there exists $0<\lambda_1<\lambda_0$ such that
$$K_{\alpha,\beta}(\phi_{n,\lambda_1})=0,\; H_{\alpha,\beta}(\phi_{n,\lambda_1})\leq H_{\alpha,\beta}(\phi_{n,\lambda_0})\leq H_{\alpha,\beta}(\phi_{n}).$$
Denoting $\phi_n$ instead of $\phi_{n,\lambda_1}$, we have $K^Q(\phi_n)=-K^N(\phi_n)\rightarrow0$ as $n$ tends to infinity, which contradicts Lemma \ref{lm1}.\\
\underline{Second case $\alpha=0$.} We assume without loss of generality that $\beta=1$ and we denote $m:=m_{0,1}$, $H:=H_{0,1}$ and $K:=K_{0,1}$. Let $(\phi_n)$ a minimizing sequence, namely
\begin{equation}\label{a}0\neq\phi_n\in \Sigma,\;K(\phi_n)= 0\;\mbox{and}\;\lim_nH(\phi_n)=\lim_nS(\phi_n)=m.\end{equation}
With the definition of $H$, $\phi_n$ is bounded in $H^1_{rd}$. Moreover, since
\begin{gather*}
H(\phi)=\frac12\Big[2\|\nabla\phi\|^2+\|\phi\|^2+(\frac\mu2-1)\int|x|^\mu G(|\phi|^2)\,dx\Big],\\
K(\phi)=2\Big[\|\phi\|^2+2\|x\phi\|^2-(\frac\mu2+1)\int|x|^\mu G(|\phi|^2)\,dx\Big],
\end{gather*}
$\phi_n$ is bounded in $\Sigma$. Let so $\phi_n\rightharpoonup\phi\;\mbox{in}\; \Sigma\;\mbox{and}\;\phi_n\rightarrow\phi\;\mbox{in}\;L^p,\;\forall p\in (2,\infty)$. Now, since for some $p\geq1+2\mu,a>0$, we have $|G(r^2)|\lesssim |r|^p(e^{ar^2}-1),\;\forall r\in\R.$ Then, for $\lambda=0^+$, $|K^N(\lambda\phi)|=o(K^Q(\lambda\phi))=\lambda^2K^Q(\phi).$ Thus
\begin{equation}\label{k}
K(\phi)<0\Rightarrow\exists\lambda\in (0,1)\quad\mbox{s. t}\quad K(\lambda\phi)=0\quad\mbox{and}\quad H(\lambda\phi)\leq H(\phi).
\end{equation}
By the lower semi-continuity of the $\Sigma$ norm via \eqref{k}, $\phi$ satisfies \eqref{a}. 
Arguing as in the case $\alpha\neq0$, yilelds$$\phi\neq 0.$$
We have $m=H(\phi)\geq\|\nabla\phi\|^2>0$. Now, with a Lagrange multiplicator $\eta$, we have $S'(\phi)=\eta K'(\phi)$. Then $-\Delta\phi=(2\eta-1)(\phi-|x|^\mu g(\phi))+(4\eta-1)|x|^2\phi-\eta\mu|x|^\mu g(\phi)$. So
\begin{eqnarray*}
\|\nabla\phi\|^2
&=&\left\langle (2\eta-1)(\phi-|x|^\mu g(\phi))+(4\eta-1)|x|^2\phi-\eta\mu|x|^\mu g(\phi),\phi\right\rangle\\
&=&(2\eta-1)(\frac{K}2-2\|x\phi\|^2+(1+\frac\mu2)\int |x|^\mu G(|\phi|^2)\,dx)\\
&+&(4\eta-1)\|x\phi\|^2+(1-2\eta(1+\frac\mu2))\int|x|^\mu |\phi|g(|\phi|)\,dx\\
&=&(2\eta-1)(1+\frac\mu2)\int|x|^\mu G(|\phi|^2)\,dx+(1-2\eta(1+\frac\mu2))\int|x|^\mu |\phi|g(|\phi|)\,dx+\|x\phi\|^2\\
&=&(2\eta-1)(1+\frac\mu2)\int|x|^\mu[ G(|\phi|^2)-|\phi|g(|\phi|)]\,dx-\frac\mu2\int|x|^\mu |\phi|g(|\phi|)\,dx+\|x\phi\|^2\\
&=&(1-2\eta)(1+\frac\mu2)\int|x|^\mu[D-1]G(|\phi|^2)\,dx-\frac\mu2\int|x|^\mu |\phi|g(|\phi|)\,dx+\|x\phi\|^2.
\end{eqnarray*}
Thus
$$m=(1-2\eta)(1+\frac\mu2)\int|x|^\mu[D-1]G(|\phi|^2)\,dx>0.$$
The fact that $(D-1)G>0$ yields that $2\eta-1<0$. Moreover, since
$$-\Delta\phi^{\lambda}=e^{-2\lambda}\Big[(2\eta-1)\phi^{\lambda}-(2\eta-1+\eta\mu)e^{\lambda\mu}|x|^\mu g(\phi^{\lambda})\Big]+(4\eta-1)|x|^2\phi^\lambda.$$
Taking a positive real $\lambda$ such that $e^{-2\lambda}(2\eta-1)=-1$, we have
$$-\Delta\phi^{\lambda}+\phi^\lambda+(1-4\eta)|x|^2\phi^\lambda=(1-\eta\mu e^{\mu-2\lambda})|x|^\mu g(\phi^{\lambda}):=a^2|x|^\mu g(\phi^{\lambda}).$$
With the scaling $\phi_c:=\phi(\frac.c)$, $c:=a^{\frac2{2+\mu}}>0$, we have
$$-c^2\Delta\phi^{\lambda}_c+\phi^\lambda_c+\frac{1-4\eta}{c^2}|x|^2\phi^\lambda_c-a^2c^{-\mu}|x|^\mu g(\phi^{\lambda}_c)=0.$$
Thus
$$-\Delta\phi^{\lambda}_c+\frac1{c^2}\phi^\lambda_c+\frac{1-4\eta}{c^4}|x|^2\phi^\lambda_c-|x|^\mu g(\phi^{\lambda}_c)=0.$$
This concludes the proof.
\section{Invariant sets and applications}
This section is devoted to prove either global well-posedness or finite time blow up of the solution to \eqref{eq1} with data in some stable sets. In all this section, we assume that $\epsilon=1$ and \eqref{4} is satisfied. Our aim is to prove Theorem \ref{sch}. Denote the quantities
\begin{gather*}
m=m_{1,-1},\quad K=K_{1,-1}=2\|\nabla v\|^2-2\|xv\|^2-2\int|x|^\mu\Big[|v|g(|v|)-(1+\frac\mu2) G(|v|^2)\Big]\,dx,\\
K^N(v):=-2\int|x|^\mu\Big[|v|g(|v|)-(1+\frac\mu2) G(|v|^2)\Big]\,dx,\quad K^Q(v):=2\|\nabla v\|^2-2\|xv\|^2,\\
T(v):=(S-\frac{K}2)(v)=\|v\|^2+2\|xv\|^2+\int|x|^\mu(D-2-\frac\mu2)G(|v|^2)\;dx.
\end{gather*}
First, let us prove existence of a ground state to \eqref{gs} for $(\alpha,\beta)=(1,-1)$. Precisely
\begin{prop}\label{tgss}
Let $(\alpha,\beta)=(1,-1)$. Assume that $g$ satisfies \eqref{4} with $[$\eqref{sc} or \eqref{cc}$]$. So, there is a minimizer of \eqref{min}, which is the energy of some solution to \eqref{gs}.
\end{prop}
The proof is based on some intermediary results. Following the proof of Lemma \ref{lm1}, we have
\begin{lem}\label{lm1'}
Let $0\neq\phi_n$ a bounded sequence of $\Sigma$ such that $\displaystyle\lim_nK^Q(\phi_n)=0$. Then, there exists $n_0\in\N$ such that $K(\phi_n)>0$ for all $n\geq n_0.$
\end{lem}
The next intermediate result of this section reads
\begin{lem}\label{dec}
For $\phi\in \Sigma$, the following real function is increasing on $\R_+$,
$$\lambda\mapsto T(\lambda\phi),$$
\end{lem}
\begin{proof}
Compute $T(\lambda\phi)=\lambda^2(\|\phi\|^2+2\|x\phi\|^2)+\int|x|^\mu(D-2-\frac\mu2)G(\lambda^2|\phi|^2)\;dx$. Compute he derivative
\begin{eqnarray*}
\frac12\partial_\lambda T(\lambda\phi)
&=&\lambda\|\phi\|^2+2\lambda\|x\phi\|^2+\lambda\int |x|^\mu|v|^2\Big[\lambda^2|v|^2G''-(1+\frac\mu2)G'\Big](\lambda^2|v|^2)\;dx\\
&=&\lambda\|\phi\|^2+2\lambda\|x\phi\|^2+\frac1\lambda\int |x|^\mu\Big[D^2-(2+\frac\mu2)D\Big]G(\lambda^2|v|^2)\;dx\\
&=&\lambda\|\phi\|^2+2\lambda\|x\phi\|^2+\frac1\lambda\int |x|^\mu\Big[(D-2-\frac\mu2)(D-1-\frac\mu2)+(1+\frac\mu2)(D-2-\frac\mu2)\Big]G(\lambda^2|v|^2)\;dx.
\end{eqnarray*}
The proof is ended via \eqref{4}.
\end{proof}
The following result will be useful.
\begin{prop}\label{enf}
We have
$$m=\displaystyle\inf_{0\neq\phi\in\Sigma}\{T(\phi),\quad K(\phi)\leq0\}.$$
\end{prop}
\begin{proof}
Let $m_1$ be the right hand side, then it is sufficient to prove that $m\leq m_1$. Take $\phi\in \Sigma$ such that $K(\phi)<0$. By Lemma \ref{lm1'}, the facts that $\displaystyle\lim_{\lambda\rightarrow0}K^Q(\lambda\phi)=0$ and $\lambda\mapsto T(\lambda\phi)$ is increasing, there exists $\lambda\in(0,1)$ such that $K(\lambda\phi)=0$ and $T(\lambda\phi)\leq T(\phi)$. The proof is closed.
\end{proof}
\begin{proof}[\bf Proof of proposition \ref{tgss}]. Let $(\phi_n)$ be a minimizing sequence, namely
$$0\neq\phi_n\in \Sigma,\;K(\phi_n)= 0\quad\mbox{and}\quad\lim_nS(\phi_n)=m.$$
 Then
$$\|\nabla\phi_n\|^2-\|x\phi_n\|^2=\int|x|^\mu(D-1-\frac\mu2)G(|\phi_n|^2)\,dx\quad\mbox{and}\quad\Big(\|\phi_n\|_\Sigma^2-\int |x|^\mu G(|\phi_n|^2)\,dx\Big)\rightarrow m.$$
So, for any real number $a\neq 0$,
$$\Big((1-a)\|\nabla\phi_n\|^2+(1+a)\|x\phi_n\|^2+\|\phi_n\|^2+a\int|x|^\mu[D-1-\frac\mu2-\frac1a]G(|\phi_n|^2)\,dx\Big)\rightarrow m.$$
Taking $a:=\frac1{1+\varepsilon_g}$, yields
$$\Big(\frac{\varepsilon_g}{1+\varepsilon_g}\|\nabla\phi_n\|^2+\frac{2+\varepsilon_g}{1+\varepsilon_g}\|x\phi_n\|^2+\|\phi_n\|^2+\frac1{1+\varepsilon_g}\int[D-2-\frac\mu2-\varepsilon_g]G(|\phi_n|^2)\,dx\Big)\rightarrow m.$$
We conclude, via \eqref{4} that $(\phi_n)$ is bounded in $\Sigma$. Taking account of the compact injection of the radial Sobolev space $H^1_{rd}$ on the Lebesgue space $L^p$ for any $2<p<\infty$, we take
$$\phi_n\rightharpoonup\phi\quad\mbox{in}\quad \Sigma\quad\mbox{and}\quad\phi_n\rightarrow\phi\quad\mbox{in}\quad L^p,\quad\forall p\in (2,\infty).$$
Assume, by contradiction, that $\phi=0$. Following the proof of Lemma \ref{lm1}, we have $K^N(\phi_n)\rightarrow0$. By Lemma \ref{lm1'}, $K(\phi_n)>0$ for large $n$ which is absurd. So$$\phi\neq 0.$$
With lower semi continuity of $\Sigma$ norm, we have $K(\phi)\leq0$ and $S(\phi)\leq m$. If $K(\phi)<0$ then by the facts that $\displaystyle\lim_{\lambda\rightarrow0}K^Q(\lambda\phi)=0$ and $\lambda\mapsto T(\lambda\phi)$ is increasing, there exists $\lambda\in(0,1)$ such that $K(\lambda\phi)=0$ and $S(\lambda\phi)=T(\lambda\phi)\leq T(\phi)\leq m$. Thus, we can assume that
$K(\phi)=0$ and $S(\phi)\leq m.$ So that $\phi$ is a minimizer satisfying
$$0\neq\phi\in\Sigma,\quad K(\phi)=0\quad\mbox{ and }\quad S(\phi)=m.$$
This implies via the assumption \eqref{4} that
 $$0<\|\phi\|^2\leq \|\phi\|^2+2\|x\phi\|^2+\int|x|^\mu(D-2-\frac\mu2)G(|\phi|^2)\;dx=2m.$$
Now, there is a Lagrange multiplier $\eta\in\R$ such that $S'(\phi)=\eta K'(\phi)$. Denoting $\mathcal L:=\mathcal L_{1,-1}$, recall that $\mathcal L(\phi):=\partial_{\lambda}(e^{\lambda}\phi(e^\lambda.))_{|\lambda=0}:=(\partial_{\lambda}\phi^\lambda)_{|\lambda=0}$ and $\mathcal L S(\phi):=(\partial_{\lambda}S(\phi^{\lambda}))_{|\lambda=0}$. We have
$$0=K(\phi)=\mathcal LS(\phi)=\langle S'(\phi),\mathcal L(\phi)\rangle=\eta\langle K'(\phi),\mathcal L(\phi)\rangle=\eta\mathcal LK(\phi)=\eta\mathcal L^2S(\phi).$$
With a computation and taking account of \eqref{4},
\begin{eqnarray*}
-(\mathcal L+2)(\mathcal L-2)S(\phi)
&=&4\int |x|^\mu(D-2-\frac\mu2)(D-\frac\mu2)G(|\phi|^2)\;dx>0.
\end{eqnarray*}
Thus, $-\mathcal L^2S(\phi)+4S(\phi)>0$, so $\eta=0$ and $S'(\phi)=0$. Finally, $\phi$ is a ground state.
\end{proof}
The last auxiliary result of this section reads
\begin{prop}\label{stb3}
Let $(\alpha,\beta)\in\R_+^*\times\R_+\cup\{(1,-1)\}$. Then,
\begin{enumerate}
\item
$m_{\alpha,\beta}$ is independent of $(\alpha,\beta)$.
\item
The sets $A_{\alpha,\beta}^+$ and $A_{\alpha,\beta}^-$ are invariant under the flow of \eqref{eq1}.
\item
The sets $A_{\alpha,\beta}^+$ and $A_{\alpha,\beta}^-$ are independent of $(\alpha,\beta)$.
\end{enumerate}
\end{prop}
\begin{proof}
Let $(\alpha,\beta)$ and $(\alpha',\beta')$ in $\R_+^2-\{(0,0)\}$.
\begin{enumerate}
\item
Let $\phi$ a minimizing of \eqref{min}, given by Thereom \ref{tgs}, then by Proposition \ref{poh}, $K_{\alpha',\beta'}(\phi)=0$. Thus $m_{\alpha,\beta}\leq m_{\alpha',\beta'}$. With the same way, we have the opposite inequality.
\item
Let $u_0\in A_{\alpha,\beta}^+$ and $u\in C_{T^*}(\Sigma)$ the maximal solution to \eqref{eq1}. Assume that for some time $t_0\in(0,T^*)$, $u(t_0)\notin A_{\alpha,\beta}^+$. Since the energy is conserved, $K_{\alpha,\beta}(u(t_0))\leq0$. So, with a continuity argument, there exists a positive time $t_1\in(0,t_0)$ such that $K_{\alpha,\beta}(u(t_1))=0.$  This contradicts the definition of $m_{\alpha,\beta}$. The proof is similar in the case of $A_{\alpha,\beta}^-$.
\item
 By the first point, the reunion $A_{\alpha,\beta}^{+}\cup A_{\alpha,\beta}^{-}$ is independent of $(\alpha,\beta)$. So, it is sufficient to prove that $A_{\alpha,\beta}^{+}$ is independent of $(\alpha,\beta)$. If $S(\phi)<m$ and $K_{\alpha,\beta}(\phi)=0$, then $\phi=0$. So $A_{\alpha,\beta}^{+}$ is open. The rescaling $\phi^\lambda:=e^{\alpha\lambda}\phi(e^{-\beta\lambda}.)$ implies that a neighborhood of zero is in $A_{\alpha,\beta}^{+}$. Moreover, this rescaling with $\lambda\rightarrow-\infty$ gives that $A_{\alpha,\beta}^{+}$ is contracted to zero and so is connected. Now, by the definition, $A_{\alpha,\beta}^{-}$ is open, and $0\in A_{\alpha,\beta}^{+}\cap A_{\alpha',\beta'}^{+}$. Since we cannot separate $A_{\alpha,\beta}^{+}$ in a reunion of $A_{\alpha',\beta'}^{+}$ and $A_{\alpha',\beta'}^{-}$. We have $A_{\alpha,\beta}^{+}=A_{\alpha',\beta'}^{+}$. The proof is achieved.
\end{enumerate}
\end{proof}
Finally, we are ready to prove the main result of this section.
\begin{proof}[\bf Proof of Theorem \ref{sch}]
There are two steps.
\begin{enumerate}
\item
With a time translation, we can assume that $t_0=0$. Thus, $S(u_0)<m$ and with Proposition \ref{stb3}, $u(t)\in A_{\alpha,\beta}^-$ for any $t\in[0,T^*)$. By contradiction assume that $T^*=\infty$. With the Virial identity via Proposition \ref{stb3} yields
$$\frac18(\|xu(t)\|^2)''(t)=\|\nabla u\|^2-\|xu\|^2-\int|x|^\mu\Big(\bar ug(u)-(1+\frac\mu2)G(|u|^2)\Big)dx=\frac12K(u(t))<0.$$
We infer that there exists $\delta>0$ such that $K(u(t))<-\delta$ for large time. Else, there exists a sequence of positive real numbers $t_n\rightarrow+\infty$ such that $K(u(t_n))\rightarrow0.$ By Proposition \ref{enf},
$$m\leq(S-\frac12K)(u(t_n))=S(u_0)-\frac12K(u(t_n))\rightarrow S(u_0)<m.$$
This absurdity finishes the proof of the claim. Thus $(\|xu\|^2)''<-8\delta$. Integrating twice, $\|xu(t)\|$ becomes negative for some positive time. This absurdity closes the proof.
\item
By Proposition \ref{stb3}, $u(t)\in A_{\alpha,\beta}^+$  for any $t\in[0,T^*)$. Moreover,
\begin{eqnarray*}
m
&\geq&(S-\frac12K_{1,-1})(u)\\
&=&\|\nabla u\|^2+2\|xu\|^2+\int|x|^\mu(D-2+\frac\mu2)G(|u|^2)\\
&\geq&\|\nabla u\|^2+\|xu\|^2.
\end{eqnarray*}
Since the $L^2$ norm of $u$ is conserved, $u(t)$ is bounded in $\Sigma$. Precisely, 
$$\sup_{0\leq t\leq T^*}\| u(t)\|_\Sigma<\infty.$$
Thus $T^*=\infty$. This ends the proof.
\end{enumerate}
\end{proof}
 \section{Instability}
In this section, we prove Theorem \ref{t2}. Precisely, under a sufficient condition, an instability result about the standing wave associated to \eqref{eq1} holds. We assume along this section, that \eqref{4} is satisfied. We denote $P:=\frac12K_{1,-1}$,  $K:=K_{1,0}$ and $\phi$ a ground state. 
\begin{defi}
For $\varepsilon>0$, we define
\begin{enumerate}
\item
the set
$$V_\varepsilon(\phi):=\{v\in\Sigma,\quad\mbox{s. t}\quad \inf_{t\in\R}\|v-e^{it}\phi\|_\Sigma<\varepsilon\}.$$
\item
For $u_0\in V_\varepsilon(\phi)$ and $u$ the solution to \eqref{eq1},
$$T_\varepsilon(u_0):=\sup\{T>0,\quad\mbox{s. t}\quad u(t)\in V_\varepsilon(\phi),\quad\mbox{for any}\quad t\in[0,T)\}.$$
\item
the set
$$\Pi_\varepsilon(\phi):=\{v\in V_\epsilon(\phi),\quad\mbox{s. t}\quad E(v)<E(\phi),\quad\|v\|\leq\|\phi\|\quad\mbox{and}\quad P(v)<0\}.$$
\end{enumerate}
\end{defi}
First, let do some computations.
\begin{prop}\label{p0}
For $v\in\Sigma,\lambda\in\R$ and $v_\lambda:=\lambda v(\lambda.)$, we have
\begin{gather*}
E(v_{\lambda})=\lambda^2\|\nabla v\|^2+\lambda^{-2}\|xv\|^2-\lambda^{-2-\mu}\int|x|^\mu G(|\lambda v|^2)\,dx,\\
\partial_{\lambda}E(v_{\lambda})=2\lambda\|\nabla v\|^2-2\lambda^{-3}\|xv\|^2-2\lambda^{-3-\mu}\Big(\int |x|^\mu\Big[\lambda |v|g(\lambda|v|)-(1+\frac\mu2)G(|\lambda v|^2)\Big]\,dx\Big),\\
P(v)=\frac12\partial_{\lambda}E(v_{\lambda})_{|\lambda=1}=\|\nabla v\|^2-\|xv\|^2-\int\Big(|v|g(|v|)-(1+\frac\mu2)G(|v|^2)\Big)\,dx,\\
\frac12\partial_{\lambda}^2E(v_{\lambda})=\|\nabla v\|^2+3\lambda^{-4}\|xv\|^2+(3+\mu)\lambda^{-4-\mu}\int|x|^\mu\Big(\lambda |v|g(\lambda|v|)-(1+\frac\mu2) G(|\lambda v|^2)\Big)dx\\
-\lambda^{-3-\mu}\int|x|^\mu\Big(\lambda |v|^2g'(\lambda|v|)-(1+\mu)|v|g(\lambda|v|)\Big)\,dx,\\
\frac12\partial_{\lambda}^2E(v_{\lambda})_{|\lambda=1}=\|\nabla v\|^2+3\|xv\|^2-\int|x|^\mu\Big(|v|^2g'(|v|)-(4+2\mu)|v|g(|v|)+(1+\frac\mu2)(3+\mu)G(|v|^2)\Big)\,dx.
\end{gather*}
\end{prop}
As a consequence,
\begin{prop}\label{p1}
Let $\phi\in\Sigma$ a solution to \eqref{gr}, then
\begin{enumerate}
\item
$\frac12\partial_{\lambda}^2E(\phi_{\lambda})_{|\lambda=1}=4\|\nabla\phi\|^2-\int|x|^\mu[|\phi|^2g'(|\phi|)-(1+2\mu)|\phi|g(|\phi|)+\mu(1+\frac\mu2) G(|\phi|^2)]\,dx.$
\item
$\frac12\partial_{\lambda}^2E(\phi_{\lambda})_{|\lambda=1}=-2\|\phi\|^2-\int|x|^\mu[|\phi|^2g'(|\phi|)-(5+2\mu)|\phi|g(|\phi|)+(2+\mu)(1+\frac\mu2) G(|\phi|^2)]\,dx.$
\end{enumerate}
\end{prop}
\begin{proof}
The first point follows because $P(\phi)=0$. The second point is a consequence of Proposition \ref{poh}, since
$$\frac12K_{2,-1}(\phi)=2\|\nabla \phi\|^2+\|\phi\|^2-\int|x|^\mu\Big[2|\phi|g(|\phi|) -(1+\frac\mu2)G(|\phi|^2)\Big]\,dx=0.$$
\end{proof}
The proof of Theorem \ref{t2} is based on several lemmas.
\begin{lem}
Assume that $\partial_{\lambda}^2E(\phi_{\lambda})_{|\lambda=1}<0$. Then, there exist two real numbers $\varepsilon_0>0,\sigma_0>0$ and a mapping $\lambda:V_{\varepsilon_0}(\phi)\rightarrow (1-\sigma_0,1+\sigma_0)$ such that $K(v_{\lambda}):=K(v_{\lambda(v)})=0$ for any $v\in V_{\varepsilon_0}(\phi)$.
\end{lem}
\begin{proof}
We have $\frac\partial{\partial\lambda}K(v^\lambda)_{|\lambda=1,v=\phi}=\left\langle K'(\phi),\partial_{\lambda}(\phi_{\lambda})_{|\lambda=1}\right\rangle$. If $\left\langle K'(\phi),\partial_{\lambda}(\phi_{\lambda})_{|\lambda=1}\right\rangle=0$, then $\partial_{\lambda}(\phi_{\lambda})_{|\lambda=1}$ would be the tangent to $\{0\neq v\in\Sigma,\quad K(v)=0\}$ at $\phi$. Therefore $\left\langle S''(u)\partial_{\lambda}(\phi_\lambda)_{|\lambda=1},\partial_{\lambda}(\phi_\lambda)_{|\lambda=1}\right\rangle\geq 0$ because $\phi$ is a minimizer. This contradicts the fact that
$$\partial_{\lambda}^2E(\phi_\lambda)_{|\lambda=1}=\partial_{\lambda}^2S(\phi_\lambda)_{|\lambda=1}=
\left\langle S''(\phi)\partial_{\lambda}(\phi_\lambda)_{|\lambda=1},\partial_{\lambda}(\phi_\lambda)_{|\lambda=1}\right\rangle
<0.$$
So, $\partial_\lambda K(v_\lambda)_{|\lambda=1,v=\phi}\neq0$ and $K(v_\lambda)_{|\lambda=1,v=\phi}=0$. A direct application of implicit theorem concludes the proof.
\end{proof}
The next auxiliary result reads.
\begin{lem}
Assume that $\partial_{\lambda}^2E(\phi_\lambda)_{|\lambda=1}<0$. Then, there exist two real numbers $\varepsilon_1>0,\sigma_1>0$ such that for any $v\in V_{\varepsilon_1}(\phi)$ satisfying $\|v\|\leq \|\phi\|$, we have
$$E(\phi)<E(v)+(\lambda-1)P(v),\quad\mbox{for some}\quad\lambda\in(1-\sigma_1,1+\sigma_1).$$
\end{lem}
\begin{proof}
With a continuity argument, there exist $\varepsilon_1>0$ and $\sigma_1>0$ such that
$$\partial_{\lambda}^2E(v_{\lambda})<0,\quad\forall(\lambda,v)\in(1-\sigma_1,1+\sigma_1)\times V_{\varepsilon_1}(\phi).$$
Thus, with Taylor expansion
$$E(v_{\lambda})<E(v)+(\lambda-1)P(v),\quad\forall(\lambda,v)\in(1-\sigma_1,1+\sigma_1)\times V_{\varepsilon_1}(\phi).$$
By the previous Lemma,
$$\forall v\in V_{\varepsilon_1}(\phi),\quad \exists \lambda\in(1-\sigma_1,1+\sigma_1)\quad\mbox{ s. t }\quad K(v_{\lambda})=0.$$
Thus, $\forall v\in V_{\varepsilon_1}(\phi)$ there exists $\lambda\in(1-\sigma_1,1+\sigma_1)$ such that
$$S(v_{\lambda})\geq S(\phi).$$
It follows that
\begin{eqnarray*}
E(v_\lambda)
&=&S(v_\lambda)-M(v_\lambda)\\
&\geq&S(\phi)-M(v_\lambda)\\
&\geq&S(\phi)-M(\phi)=E(\phi).
\end{eqnarray*}
The proof is finished.
\end{proof}
\begin{lem}
Assume that $\partial_{\lambda}^2E(\phi_\lambda)_{|\lambda=1}<0$. Then, for $u_0\in \Pi_{\varepsilon_1}$ there exists a real number $\sigma_0>0$ such that the solution $u$ to \eqref{eq1} satisfies
$$P(u(t))<-\sigma_0,\quad\mbox{for all}\quad t\in [0,T_{\varepsilon_1}(u_0)).$$
\end{lem}
\begin{proof}
Let $u_0\in \Pi_{\varepsilon_1}$, then $E(u_0)<E(\phi),\|u_0\|\leq\|phi\|$ and $P(u_0)<0$. Put $\sigma_2:=E(\phi)-E(u_0)>0.$ With the previous Lemma, there exists $\lambda\in(1-\sigma_1,1+\sigma_1)$ such that
$$(\lambda-1)P(u(t))+E(u(t))>E(\phi),\quad\forall t\in[0,T_{\varepsilon_1}(u_0)).$$
By conservation of the energy, there exists $\lambda\in(1-\sigma_1,1+\sigma_1)$ such that
$$(\lambda-1)P(u(t))>\sigma_2,\quad\forall \,\,t\in[0,T_{\varepsilon_1}(u_0)).$$
So, by continuity argument via $P(u_0)<0$, we have $P(u(t))<0$ for all $t\in[0,T_{\varepsilon_1}(u_0))$. Then, $\lambda-1<0$ and $-\sigma_0:=-\frac{\sigma_2}{1-\lambda}<0$ for any $t\in[0,T_{\varepsilon_1}(u_0))$.
The proof is closed.
\end{proof}
Now, we are ready to prove the crucial result of this section.
\begin{proof}[\bf Proof of Theorem \ref{t2}]
By Propositions \ref{p0} and \ref{p1}, it follows that
$$\frac12\partial_{\lambda}^2E(\phi_\lambda)_{|\lambda=1}<0,\quad\partial_{\lambda}E(\phi_\lambda)=\frac1\lambda P(\phi_\lambda).$$
Then, $1$ is a maximum for $\lambda\mapsto E(\phi_\lambda)$ and $\frac1\lambda P(\phi_\lambda)<P(\phi)=0$ as $\lambda>1$ near one (we denote $\lambda=1^+$). Thus, $\phi_\lambda\in\Pi_\varepsilon$ for $\varepsilon=\varepsilon(\lambda)>0$ and $\lambda=1^+$. Take $u_0=\phi_\lambda$, for $\lambda=1^+$, then
$$u_0\in\Pi_\varepsilon\quad\mbox{and}\quad \lim_{\lambda=1}\|u_0-\phi\|_\Sigma=0.$$
By the previous Lemma, there exists $\sigma_0>0$ such that
$$P(u(t))<-\sigma_0\quad\mbox{for all}\quad t\in[0,T_{\varepsilon_1}(u_0)).$$
Now, if $e^{it}\phi$ is orbitally stable, $T_{\varepsilon_1}(u_0)=\infty$ and $P(u)<-\sigma_0$ on $\R_+$. With virial identity, $\|xu\|$ becomes negative for long time. This absurdity finishes the proof.
\end{proof}

\end{document}